\documentclass{amsart}

\usepackage{amsfonts,amssymb,amsmath,amsthm,enumerate,latexsym,color}
\usepackage[backref=page, colorlinks=true, citecolor=cyan, linkcolor=blue]{hyperref}
\usepackage{hyperref,comment,graphicx}

\newtheorem{theorem}{Theorem}[section]
\newtheorem{lem}{Lemma}[section]
\newtheorem{remark}{Remark}[section]
\numberwithin{equation}{section}
\newtheorem*{theorem*}{Theorem}

\newtheorem{defi}{Definition}[section]

\newtheorem{prop}{Proposition}
\newtheorem{step}{Step}

\newcommand{\R}{\mathbb{R}}

\newcommand{\be}{\begin{equation}}
\newcommand{\ee}{\end{equation}}

\newcommand{\sx}{{\Delta}}

\newcommand{\eps}{\varepsilon}
\newcommand{\p}{\partial}

\newcommand{\Rd}{{\mathbb{R}^d}}
\newcommand{\Prob}{\mathcal{P}}
\newcommand{\T}{\mathcal{T}}
\newcommand{\F}{\mathcal{F}}

\newcommand{\K}{\mathcal{K}}

\newcommand{\E}{\mathbb{E}}

\newcommand{\dm}{ h }

\newcommand{\pp}{\bar p}

\begin{document}

\subjclass{35Q91, 35F20, 91A26}

%\title{Evolutionnary game theory: from microscopic interactions to mean-field equations.}
\title[Evolutionary game theory in mixed strategies]
{Evolutionary game theory in mixed strategies: from microscopic interactions to kinetic equations.}

\author{Juan Pablo Pinasco}
\address{IMAS (UBA-CONICET) and Departamento de Matem\'atica, Facultad de Ciencias Exactas y Naturales,
 Universidad de Buenos Aires, Ciudad Universitaria, 1428 Buenos Aires, Argentina}
\email{jpinasco@dm.uba.ar, jpinasco@gmail.com}

\author{Mauro Rodriguez-Cartabia}
\address{IMAS (UBA-CONICET) and Departamento de Matem\'atica, Facultad de Ciencias Exactas y Naturales, Universidad de Buenos Aires, Ciudad Universitaria, 1428 Buenos Aires, Argentina}
\email{mrodriguezcartabia@gmail.com}

\author{Nicolas Saintier}
\address{IMAS (UBA-CONICET) and Departamento de Matem\'atica, Facultad de Ciencias Exactas y Naturales, Universidad de Buenos Aires, Ciudad Universitaria, 1428 Buenos Aires, Argentina}
\email{nsaintie@dm.uba.ar}

\thanks{This work was partially supported by Universidad de Buenos Aires under grants
20020170100445BA, by Agencia Nacional de Promocion Cientifica y Tecnica PICT 201-0215 and PICT 2016 1022.}

\begin{abstract}
In this work we propose a kinetic formulation for evolutionary game theory for zero sum games when the agents use mixed strategies. We
start with a simple adaptive rule, where after an encounter each agent increases the probability of play
the successful pure strategy used in the match. We derive the Boltzmann equation which describes the macroscopic effects of
this microscopical rule, and we
obtain a first order, nonlocal, partial differential equation as the limit when the probability change goes to zero.

We study the relationship between this equation and the well known replicator equations, showing the equivalence between the
concepts of Nash equilibria, stationary solutions of the partial differential equation, and the equilibria  of the replicator equations. Finally,
we relate the long time behavior of solutions to the partial differential equation and the stability
of the replicator equations.
\end{abstract}

\keywords{kinetic models, evolutionary game theory, mean field games}

\maketitle

%\tableofcontents

\section{Introduction}

Evolutionary game theory, introduced in the 70s by Jonker and Taylor \cite{taylor1978evolutionary} and Maynard Smith \cite{smith1982evolution},
 is a beautiful mix of biology and game theory.
Each player in some  population interact repeatedly with other players by playing a game,
obtaining a pay-off as reward or punishment of the pure strategies or actions they used. We
can consider now an {\it evolutive} mechanism, where agents with higher pay-offs have more
offsprings, or an {\it adaptive} one, where less successful players change strategies and
imitate the ones which performed better. This process is mathematically formalized using
systems of ordinary differential equations or difference equations. Essentially, we have an
equation describing the evolution of the proportion of players in each pure strategy,
 usually in the form of a rate equation where the growth rate is proportional to the {\it fitness} of the strategy.
The so-called replicator equations (see \eqref{ReplicatorSystemIntro} below)  is a famous and well-studied example of such systems. However,
observe that the fitness also evolves, since it depends on the distribution of the population on the strategies, so
interesting problems appears concerning  the existence and  stability of  fixed points, and their
relationship with the Nash equilibria of the underlying game.
We refer the interested reader to the monographs \cite{cressman2013stability,hofbauer1998evolutionary,sandholm2010population} for details.

Similar mechanisms when players use mixed strategies are less frequent in the literature.
Recently, in  \cite{pinasco2018game} we introduced an evolutive model for finitely many players, and we obtained
a systems of ordinary differential equations describing the evolution of the mixed strategy of each agents.
The number of equations is then proportional to the number of players.
Thus, the limit of infinitely many players seems to be intractable in this way.
However, a simple hydrodynamic interpretation is possible, leading to  a first order
partial differential equation modeling the strategy updates: we can think of the players as
immersed in a fluid, flowing in the simplex $\Delta$ of mixed strategies, following the drift induced by
 the gradient of the fitness of the strategies given the distribution
of  the whole population on this simplex.

\medskip

Let us present a  brief outline of our model and the main results in this work, and see Section
\S 2 for the precise definitions, notations, and previous results. We consider a population of
agents playing a finite, symmetric, zero sum game, with pay-off matrix $A$. Each player
starts with a given mixed strategy, i.e. a probability distribution on the set of pure  strategies.
They are randomly matched in pairs, and select at random  a pure strategy  using their
respective mixed strategies which they use to play the game. After the match, each player
changes its mixed strategy by adding a small quantity $h$ to the winner strategy, and
reducing the loosing one in the same amount. Some care is necessary in the definition of $h$
to avoid values greater than one, or less than zero. This is needed only near the boundary of
the simplex so we replace the constant $h$ by a function $h(p)$ satisfying
$h(p)<dist(p,\partial \Delta)$. This adaptive, microscopic rule, first introduced and analyzed
in \cite{pinasco2018game} for finitely many players, induces  a flow of the players in the
simplex, whose study is the main purpose of this paper.

Let us call $u_t^h$ the distribution of agents on the simplex.
We can think of $u_t^h(p)dp$ as the probability to find a player with an strategy $q$ in a cube of
area $dp$ centered at $p$.
Now, it is possible to describe the time evolution of $u_t^h$ with a Boltzmann-like equation,
whose collision part reflects the dynamics in the changes of strategies due to encounters.
This procedure  is strongly inspired by the kinetic theory of rarefied gases and granular flows
and has been successfully implemented
to model a wide variety of phenomena in applied sciences (see e.g.  \cite{arlotti2002generalized,bellomo2008modeling,bellomo2013complex,pareschi2013interacting,PPS,perez2018opinion} and
the surveys in Ref. \cite{naldi2010mathematical} for further details).

However, Boltzmann-like equation are challenging objects to study. Performing the so-called
grazing limit or quasi-invariant limit (see
\cite{degond1992fokker,desvillettes2001rigorous,desvillettes1992asymptotics,pareschi2013interacting}),
we can approximate it by a Fokker-Planck equation which is satisfied by $v_t= \lim_{h\to 0}
u_{t/h}^h$. Besides the intrinsic stochastic nature of the interactions, we can add a small
noise in the agents' movements. The Fokker-Planck equation then reads
$$
\frac{\partial v_t}{\partial t} + div(\mathcal{F}[v_t]v_t)
= \lambda \sum_{i,j=1}^d Q_{i,j} \frac{\partial^2 }{\partial p_i \partial p_j}(Gv_t),
$$
where $Q$ and $G$ depend on $v_t$ and the intensity of the noise, $\lambda\ge 0$ depends on the ratio of the noise to
the convection term, and the vector-field $\mathcal{F}[v_t]$, which depends on $v_t$,  is given by
$$ \mathcal{F}[v_t]=h(p)[p_ie_i^TA\bar{p}(t) + \bar{p}_i(t)e_i^TAp],$$
with $ \bar{p}(t)=\int_\Delta p\,dv_t(p)$ the mean strategy at time $t$.

In particular, when $\lambda=0$, i. e., when the convection term dominates the noise, we obtain the first order, nonlocal, mean field equation
\begin{equation}\label{TransportEquIntro}
 \frac{\partial v}{\partial t} + div(\mathcal{F}[v_t]v_t) = 0.
\end{equation}
Existence and uniqueness of solutions for  \eqref{TransportEqu} follows by using the
classical ideas of \cite{braun1977vlasov,dobrushin1979vlasov,neunzert1974approximation}
(see also \cite{canizo2011well,golse2016dynamics}).

One of the main focus of the paper is the study of the long time behavior of solutions to
\eqref{TransportEquIntro} and the stability of the stationary solutions of the form $v = \delta_p$,
where $\delta_p$ is a Dirac mass at $p$;, which corresponds to the case where all the
players use the same mixed strategy $p$.

These issues are closely related to the behaviour of the integral curves of the vector-field $\F$.
It is worth noticing the close resemblance of $\F$ with the replicator equations
\begin{equation}\label{ReplicatorSystemIntro}
 \frac{dp_i}{dt} = p_i(e_i^TAp - p^TAp).
\end{equation}
We can thus expect a relationship between the long-time behaviour of the solutions of
\eqref{TransportEquIntro} and of the solutions to the replicator equations
\eqref{ReplicatorSystemIntro}. A well-known result in evolutionary game theory, known as the
Folk Theorem (see \cite{hofbauer2003evolutionary} and Theorem \ref{FolkThm}  below),
relates the long-time behaviour of the solution of the replicator equation to the Nash
equilibria of the zero-sum game with pay-off matrix $A$.

\medskip

Our first main theorem  can be thought of as a generalization of the Folk Theorem.
Indeed we prove that  the following statements are equivalent:
\begin{itemize}
\item $p$ is a Nash equilibrium of the game.
\item $\delta_p$ is an stationary solution to \eqref{TransportEquIntro},
\item $p$ is an equilibrium of the replicator equations \eqref{ReplicatorSystemIntro},
\item $Ap=0$, where $A$ is the pay-off matrix of the game.
\end{itemize}

Our second main theorem states that if $v_t$ is a solution to \eqref{TransportEquIntro}, then
the  mean strategy of the population, $$ \bar{p} = \int_\Delta p dv_t$$ is a solution to the
replicator equations \eqref{ReplicatorSystemIntro} while it stays in $\{h=c\}$. See Section \S
5 for the precise statement of both theorems.

 Finally, we  show some results about  the asymptotic behaviour of the solution $v_t$ of
\eqref{LimitEqq} and their relationship with  the game with pay-off matrix $A$.

In the simplest case,  namely the case of a two-strategies game, we can precisely describe
the asymptotic behavior of $v_t$. Then, we turn our attention to symmetric games with an
arbitrary number of strategies. Following \cite{sandholm2010population}, a zero sum
symmetric game has no stable interior equilibria, and periodic  solutions to the replicator
equations appear as in the classical rock-paper-scissor. So, we will show that if all  the
trajectories of the replicator equations are periodic orbits, then $v_t$ is also periodic.

\bigskip

Let us compare briefly other works dealing with similar issues. To our knowledge, the first
work dealing with evolutionary game theory for mixed strategies is due to Boccabella,
Natalini and Pareschi  \cite{BNP}. They considered an integro-differential equation modeling
the evolution of a population on the simplex of mixed strategies following the idea behind
the replicator equations. That is, the population in a fixed strategy $p$ will increase
(respectively, decrease) if the expected pay-off given the distribution of agents in the
simplex is greater than (resp., lower than) the mean pay-off of the population. A full analysis
of the stability and convergence to equilibria is performed for binary games.   Let us remark
that
 their dynamics inherit several properties of the replicator equations, and if some mixed strategy $p$
 has zero mass in the initial datum, the solution  will remains equal to zero forever. So, agents cannot learn
 the optimal way of play, and they are faced to extinction or  not depending on the
mixed strategies which are present in the initial datum. In some sense, this is equivalent to
consider each mixed strategy as a pure one in a new game. The mathematical theory for
infinitely many pure strategies was developed by Oeschssler and Riedel in
\cite{oechssler2001evolutionary}, and studied later by Cressman
\cite{cressman2005stability},  Ackleh and Cleveland  \cite{cleveland2013evolutionary}, and
also Mendoza-Palacios and Hern{\'a}ndez-Lerma \cite{mendoza2015evolutionary}.

 Finally,  we can cite also \cite{albi2019boltzmann,marsan2016stochastic,salvarani2019kinetic,tosin2014kinetic} where
binary (or more general) games have pay-offs depending on some real parameters, and they
study the distribution of players on this space of parameters (say, wealth, velocity, opinion,
among many other characteristics). Now, the strategy that they play  in the binary game is
selected using partial or total knowledge of the global distribution of agents. For example,
we can assume that $x$ represent the wealth of an agent, and in the game they will
cooperate or not depending on the median of the society; if they cooperate, a fraction of
wealth is transferred from the richer to the poorer, on the other hand, if they does not
cooperate, no changes occur.

\bigskip

The paper is organized as follow. In Section \S 2 we first recall some basic results concerning
game theory, the replicator equations and measure theory to make the paper self-contained.
We then present in Section \S 3  the model we are concerned with,   and we deduce  in
Section \S 4 the partial differential equations allowing to study the long-time behaviour of the
system. In  Section \S 5 we prove the generalization of the Folk Theorem.  Section \S 6  is
devoted to the study of the asymptotic behaviour and the stability of the stationary solutions
to the mean-field equation \eqref{TransportEquIntro} and their relationship with the
replicator equations, and we present agent based and numerical solutions of the
model. The lengthy or technical proofs of existence and uniqueness of solutions of the relevant
equations are postponed to the Appendix for a better readability
of the paper.

\section{Preliminaries. }

\subsection{Preliminaries on game theory}

We briefly recall some basics about game theory and refer to the excellent  references
available on general game theory  for more details (e. g. \cite{laraki2019mathematical}).
Since we are concerned in this paper with two players, symmetric, zero-sum, finite games in
normal form, we will limit our exposition to this setting.

A two-players finite game in normal form consists of two players named I and II, two finite
sets $S^1=\{s_1,\ldots,s_{d_1}\}$ and $S^2=\{\tilde s_1,\ldots,\tilde  s_{d_2}\}$, and two
matrices $A=(a_{ij})_{1\le i\le d_1,1\le j\le d_2}$, $B=(b_{ij})_{1\le i\le d_1,1\le j\le d_2} \in
\R^{d_1\times d_2}$. The elements of $S^1$ (respectively, $S^2$) are the pure strategies of
the first (resp., second) player and models the actions he can choose. Once both players
have chosen a pure strategy each, say I chosed $s_i\in S^1$ and II chosed $\tilde s_j\in S^2$,
then I receives the pay-off $u^1(s_i,\tilde s_j):=a_{ij}$ and II receives $u^2(s_i,\tilde
s_j)=:b_{ij}$. Thus the pay-off received by each player depends on the pure strategies chosen
by each players and on their pay-off functions $u^1$ and $u^2$. It is convenient to identify
$s_i$ with the the i-th canonical vector $e_i$ of $\R^{d_1}$, and $\tilde s_j$ with the the j-th
canonical vector $e_j$ of $\R^{d_2}$. This way the pay-off of I and II are
$$ u^1(s_i,\tilde s_j)=e_i^TAe_j=a_{ij},\qquad u^2(s_i,\tilde s_j)=e_i^TBe_j=b_{ij}. $$
The game is said:
\begin{itemize}
\item symmetric, when I and II are indistinguishable both from the point of view of the sets
    of actions available and the pay-off functions: \begin{itemize} \item[]  $S^1=S^2=:S$, and
      \item[] $u^1(s,\tilde s)=u^2(\tilde s,s)$ for any $(s,\tilde s)\in S\times S$.\end{itemize}
          This last equality means that $A=B^T$.
\item zero-sum if $u^1+u^2=0$, i. e. $B=-A$.  This means that the gain of one player is
    exactly the loss of the other one.
\end{itemize}
Notice in particular that the game is symmetric and zero-sum if and only if $A^T=-A$, in other
words, the matrix $A$ is antisymmetric.

We illustrate the above definitions with the popular Rock-Paper-Scissor game.
This is a two-players zero-sum  game with pure strategies
$$S^1=S^2=\{Rock, Paper, Scissor\}=\{e_1,e_2,e_3\}$$
(we identifed as before each pure strategy with the canonical vectors of $\R^3$.)
We then define the pay-off matrix $A$ of the first player as
\begin{equation}\label{matrizRPS}
A=\left(\begin{array}{rrr} 0 & -a & b\\ b& 0& -a\\ -a & b & 0\end{array}\right)
\end{equation}
where $a$, $b >0$, and the pay-off matrix of the second player is $B=-A$ being he game
zero-sum by definition. For instance  if I plays Paper and II plays Rock then I earns $a_{21}=b$
(so II looses b), and if I plays Scissor and II Rock then I earns $a_{32}=-a$ (and thus II gains
a). When $a=b$ then  $A$ is antisymmetric meaning that the  game is symmetric.

A central concept in game theory is that of Nash equilibrium. A Nash equilibrium is a pair
$(s^*,\tilde s^*)\in S^1 \times S^2$ such that neither player has incentive to change its
action given the action of the other player: there hold at the same time
$$ u^1(s^*,\tilde s^*)\ge u^1(s,\tilde s^*) \qquad \text{for any $s\in S^1$, $s\neq s^*$, }$$
and
$$ u^2(s^*,\tilde s^*)\ge u^2(s^*,\tilde s) \qquad \text{for any $\tilde s\in S^2$,
	$\tilde s\neq \tilde s^*$. }$$ A Nash equlibrium thus models a status quo situation. When
the above inequalities are strict, the Nash equilibrim is said to be strict. However, there not
always exists a Nash equilibrium in pure strategies as can be seen for instance in the
Rock-Paper-Scissor game \eqref{matrizRPS}: by symmetry, if both players are playing
some fixed pure strategies, both have an incentive to change to another strategy.

The main mathematical issue here is the lack of convexity of the strategy spaces $S^1$ and
$S^2$. This motivates the introduction of mixed strategies as probability measures over the
set of pure strategies: a mixed strategy for I is a vector $p=(p_1,\ldots,p_{d_1})$ where $p_i$
is the probability to play the i-th pure strategy $s_i$. Thus $p_i\in [0,1]$ and $\sum_i p_i=1$.
Remember that we identify the pure strategies with the canonical vectors of $\R^{d_1}$. Let
$$\Delta_1=\{p=(p_1,\ldots,p_{d_1})\in\R^{d_1}:\, p_1,\ldots,p_{d_1}\ge 0,\, \sum_i p_i=1 \} $$
be the simplex in $\R^{d_1}$, and similarly we denote $\Delta_2$ the simplex in
$\R^{d_2}$. We extend the pay-off functions $u^1$ and $u^2$ to $\Delta_1\times \Delta_2$
as expected pay-offs in the following way: for $(p,\tilde p)\in \Delta_1\times \Delta_2$,
$$ u^1(p,\tilde p) = \sum_{i,j} p_i\tilde p_j u^1(s_i,\tilde s_j)
= \sum_{ij} p_i\tilde p_j a_{ij} = p^TA\tilde p $$ and similarly,
$$ u^2(p,\tilde p) = p^TB\tilde p.  $$

We can then extend the notion of Nash equlibrium to mixed strategies saying that
$(p^*,\tilde p^*)\in \Delta_1\times \Delta_2$ is a Nash equilibrium if  at the same time
$$ u^1(p^*,\tilde p^*)\ge u^1(p,\tilde p^*) \qquad \text{for any $p\in \Delta_1$,
	$p\neq p^*$, }$$
and
$$ u^2(p^*,\tilde p^*)\ge u^2(p^*,\tilde p) \qquad \text{for any $\tilde p\in \Delta_2$,
	$\tilde p\neq \tilde p^*$. }$$ Nash'celebrated Theorem states that a finite  game in
normal form always has a Nash equilibrium in mixed strategies, we refer to
\cite{geanakoplos2003nash} for a simple proof. Moreover, when the game is symmetric (so
that $S:=S^1=S^2$) there always exists a Nash equilibrium of the form
$(p^*,p^*)\in\Delta\times\Delta$. For instance, in the symmetric Rock-Paper-Scissor game
\eqref{matrizRPS} with $a=b$, the unique symmetric Nash equilibrium is $p^*=(1/3,1/3,1/3)$.
This means that no players has an incentive to deviate when they  choose their action with
equiprobability. Of course, for zero sum games, the existence of an equilibria goes back to
Von Neumann's Minimax Theorem.

\subsection{The replicator equations}\label{subseccionreplicador}

The concept of Nash equilibrium is quite static and computationally challenging. Various
dynamical processes have been studied to model the process of learning, the discovering a
Nash equilibrium by an individual. A popular one is a system of ordinary differential
equations known as replicator equation that was introduced by Taylor and Jonker in
\cite{taylor1978evolutionary} (see also  \cite{schuster1983replicator,smith1982evolution}).

Consider a large population of individuals randomly matched in  pairs to play a two-player
symmetric game with pay-off matrix $A\in \R^{d\times d}$. The players are divided into $d$
groups according to the pure strategy they use. Let $p(t)=(p_1(t),\ldots,p_d(t))$, where
$p_i(t)$ is the proportion of individuals playing the $i$-th pure strategy $e_i$ at time $t$. We
want to write down an equation for $p_i'(t)$, $i=1,\ldots,d$, modelling the fact that
individuals playing a strategy with high fitness should be favored and they will produce more
offsprings than individuals playing a low fitness strategy.

In the case of the replicator equation, the fitness of an individual playing the $i$-th strategy
is defined as the difference between the expected pay-off received against a randomly
selected individual, and the expected pay-off received by a randomly selected individual
playing against another randomly selected individual. Thus, the fitness of the $i$-th strategy
$e_i$ is $e_i^TAp(t)-p(t)^TAp(t)$.  Notice that the fitness depends on the distribution of
strategies in the whole population and change in time. We then assume that agents  with
positive (respectively, negative) fitness has a reproductive advantage (resp., disadvantage)
leading to their reproduction (resp., death) at a rate proportional to their fitness. We thus
arrive at the replicator equations,
\begin{equation}\label{ReplicatorSystem}
  \frac{d}{dt}p_i=p_i(e_i^TAp-p^TAp) \qquad i=1,\ldots,d.
\end{equation}
It is easily shown that if $p(0)\in \Delta$, where $\Delta$, then $p(t)\in \Delta$ for all time
$t\ge 0$.

There is a strong connection between the rest point of the system  \eqref{ReplicatorSystem}
and the Nash equilibria of the game with pay-off matrix $A$ as stated in the so-called {\it Folk
Theorem}:

\begin{theorem}\label{FolkThm}
Let us consider a two player  symmetric  game in normal form with finitely many pure
strategies. Then it hold
\begin{enumerate}
\item if $(p,p)\in\Delta\times\Delta$ is a symmetric Nash equilibrium, then $p$ is a rest
    point of \eqref{ReplicatorSystem},
\item if $(p,p)$ is a strict Nash equilibrium, then it is an asymptotically stable  rest point of
    \eqref{ReplicatorSystem},
\item if the rest point $p$ is the limit as $t\to +\infty$ of a trajectory lying in the interior of $\Delta$ then $(p,p)$ is a Nash equilibrium,
\item if the rest point $p$ is stable then $(p,p)$ is a Nash equilibrium.
\end{enumerate}
Moreover, none of the converse statement holds.
\end{theorem}

We refer to the surveys \cite{hofbauer2003evolutionary,mikekisz2008evolutionary})  for a
proof and related results on the replicator equations.

\subsection{Preliminaries on probability measures and transport equations. }

We denote by $\mathcal{M}(\sx)$ the space of bounded Borel measures on $\sx$ and by $\mathcal{P}(\sx)$ the convex
cone of probability measures on  $\sx$.
We denote by $\|.\|_{TV}$ the total variation norm on $\mathcal{M}(\sx)$ defined as
$$ \|\mu\|_{TV} = \sup_{\|\phi\|_\infty\le 1} \int_{\sx} \phi\,d\mu. $$
However, the total variation norm will be  too strong for our purpose and it will be more
convenient to work with the weak*-convergence. We say that a sequence $(\mu_n)_n\subset
\mathcal{P}(\sx)$ converges weak* to  $\mu\in \mathcal{P}(\sx)$ if
$$ \int_\sx \phi\,d\mu_n \to \int_\sx \phi\,d\mu \qquad \text{for any $\phi\in C(\sx)$. }$$
It is well-known that, since $\Delta$ is compact, that $\mathcal{P}(\sx)$ is compact for the
weak* topology. The weak*-topology can be metricized in many ways. It will be convenient to
consider the Monge-Kantorovich or Wasserstein distance on $\mathcal{P}(\sx)$ defined as
$$W_1(u,v):=\sup \, \left(\int_\sx \varphi(p)\,du(p) - \int_\sx \varphi(p)\,dv(p)\right),   $$
where the supremum is taken over all the Lipschitz functions $\varphi$ with Lipschitz
constant $Lip(\varphi)\leq 1$. It is  known that $W_1$ is indeed a distance that metricizes
the weak*-topology, see  \cite{villanioptimal}.

We will work in this paper with first order partial differential equations of the form
\begin{equation}\label{TransportEq}
 \p_t \mu_t + \text{div}(v(x)\mu_t) = 0 \qquad \text{in $\R^d$}
\end{equation}
with a given initial condition $\mu_0\in P(\R^d)$ and
where $v:\R^d\to \R^d$ is a given vector-field usually assumed to be bounded and globally Lipschitz. A solution to this equation is $\mu\in C([0,+\infty),P(\R^d))$ satisfying
$$ \int_{\R^d} \phi\,d\mu_t =  \int_{\R^d} \phi\,d\mu_0
+ \int_0^t \int_{\R^d} v(x)\nabla\phi(x)\,d\mu_s(x)ds \qquad \text{for any $\phi\in C^\infty_c(\R^d)$.}$$
Let $T_t:\R^d\to \R^d$ be the flow of $v$ defined for any $x\in\R^d$ by
\begin{eqnarray*}
\frac{d}{dt} T_t(x) & =& v(T_t(x)) \qquad \text{for any $t\in\mathbb{R}$,} \\
\qquad T_{t=0}(x)& =& x.
\end{eqnarray*}
It is well-known (see e.g. \cite{villanioptimal}) that equation \eqref{TransportEq} has a
unique solution given by $ \mu_t = T_t\sharp \mu_0$, the push-forward of $\mu_0$ by
$T_t$. This means that $\int_{\R^d} \phi\,d\mu_t= \int_{\R^d} \phi(T_t(x))\,d\mu_0(x)$ for
any $\phi$ bounded and measurable. This result, which is simply a restatement of the
standard characteristic method, can be generalized to deal with equations with a
non-autonomous vector-field $v(t,x)$ assumed to be continuous and globally Lipschitz in $x$,
uniformly in $t$.

\section{Description of the model}

We consider a population of agents. Two randomly selected agents meet, they play a game,
and then update their strategy taking into account the outcome of the game. The game
played during an interaction is always the same. It is a two-player, zero-sum game with a set
$\{s_1,\ldots, s_d\}$ of pure strategies and whose pay-off is given by a matrix
$A=(a_{lm})_{1\le l,m\le d}\in\R^{d\times d}$. We will assume the game  is symmetric i.e.
$A^T=-A$, and with out loss of generality  we take $|a_{lm}|\le 1$ for any $l,m=1, \ldots, d$.

Each agent $i$ has a mixed strategy $p=(p_1,\dots,p_d)\in\sx$. Here $p_l$ is the probability
that agent $i$ plays the $l$-th pure strategy $s_l$.

When two agents $i$ and $j$ meet and  play the game,  they update their respective
strategies using a myopic rule, both agents  increase by $\delta h(p)>0$ the probability of
playing the winning strategy and decrease by $\delta  h(p)>0$  the loosing one. Here,
$\delta$ is a small positive parameter, and $ h(p)$ is a  positive function of $p$, to ensure
that the updated strategy $p^*$ remains in $\sx$. For instance, we can take
\begin{align}\label{defdelta}
\dm(p) := \min\Big\{ \prod_{i=1}^d p_i, c\Big\}.
\end{align}
where $c<1$ is a positive constant, and hence  $\dm(p)\to 0$ as $p\to \partial \sx$.

  More precisely, if the pure strategies $s_l$
and $s_m$ were played, agent $i$ only updates the probabilities $p_l$, $p_m$ to  $p^*_{l}$,
$p^*_m$ as follows
\begin{equation}\label{UpDateProba1}
\begin {split}
  p^*_l & =
	     p_l+a_{lm}\delta h(p)   \\
  p^*_m  & =	
	     p_m -a_{lm}\delta h(p),
\end{split}
\end{equation}
Agent $j$ updates the probabilities $\tilde p_l$, $\tilde p_m$ in the same way. Notice that
probabilities are raised/lowered proportionally to the gain/loss $a_{lm}\delta$.

To model the choice made by agent $i$ of which pure strategy to play, we fix a  random
variable  $\zeta$  uniformly distributed in $[0,1]$ and then consider  the random vector $
f(\zeta; p) = ( f_1(\zeta; p),\ldots, f_d(\zeta; p))$ where
\begin{equation}\label{definicionf}
   f_i(\zeta; p):= \begin{cases}
  1&\text{ if } \sum_{j<i}p_j\leq \zeta <\sum_{j\leq i}p_j,\\
  0 & \text{  otherwise.}
\end{cases}
\end{equation}
Notice  that $f_i(\zeta; p)=1$ with probability $p_i$.
Agent $j$ fixes in the same way  a  random variable  $\tilde\zeta$  uniformly distributed in $[0,1]$.
Then $f(\zeta,p)^TA f(\tilde \zeta,\tilde p)\in [-1,1]$ is the  pay-off of $i$ when playing against $j$ (recall that the coefficient of $A$ belongs to $[-1,1]$).
We can thus rewrite the updating rule \eqref{UpDateProba1} as
\begin{equation*}
p_i^*=
  \begin{cases}
  p_i+\delta \dm(p) f(\zeta,p)^TA f(\tilde \zeta,\tilde  p)& \text{ if }  f_i(\zeta,p)=1 \text{ and } f_i(\tilde \zeta,\tilde p)=0,\\
   p_i- \delta \dm(p) f(\zeta,p)^TA f(\tilde \zeta,\tilde p) &\text{ if } f_i(\zeta,p)=0 \text{ and } f_i(\tilde \zeta,\tilde p)=1,\\
    p_i   & \text{ otherwise }
\end{cases}
\end{equation*}

We can also add a small noise to $p_i^*$
in the following way.  We fix $r>0$ small enough so that $(\delta+r)<1$, and a smooth function
 $G $ such that $G(p)\leq p_i$ for any  $p\in\sx$ and any $i$ like e.g.  $G(p)=\dm(p)$.
We then consider a uniformly distributed random vector $q$ in $ \sx$.
The additive random noise is then taken as $r(q_i-1/d)G(p)$.

\medskip

We thus arrive at the following  interaction rule:

\begin{defi} Consider an agent with strategy $p$ interacting with an agent with strategy $\tilde p$ through
the game defined  by the matrix $A$. They update their strategies from $p$ to $p^*$, and
$\tilde p $ to $\tilde p^*$, as follows
  \begin{equation}\label{interaccion}
\begin{split}
 & p^*= p+ \delta \dm(p) f(\zeta,p)^TA f(\tilde \zeta,\tilde p) \Big(f(\zeta,p)-f(\tilde \zeta,\tilde p)\Big)  +  r(q -\vec{1}/d )G(p),\\
 &  \tilde p^*= \tilde p+ \delta \dm(\tilde p) f(\zeta,p)^TA f(\tilde \zeta,\tilde p) \Big(f(\zeta,p)-f(\tilde \zeta,\tilde p)\Big)
                     +  r(\tilde q - \vec{1}/d )G(\tilde p),
   \end{split}
\end{equation}
where $\vec{1}=(1,\ldots,1)\in \Rd$.
\end{defi}

Let us remark that $p^*$ and $\tilde p^*$ are random variables. Indeed there are two
sources of randomness in the  updating rule. First, there is the presence of the random
variables $\zeta$ and $\tilde \zeta$ which model the fact that the players choose the pure
strategy they are about to play at random using their mixed strategy $p,\tilde p$. The second
factor of randomness is the  noise  $r(q -\vec{1}/d )G(p)$.

Let us  verify now that $p^*$ remains   in the simplex $ \sx$.

\begin{lem}
The  strategy $p^*$ belongs to $\sx$.
\end{lem}

\begin{proof} Starting from
$$
p^*= p+ \delta \dm(p) f(\zeta,p)^TA f(\tilde \zeta,\tilde p) \Big(f(\zeta,p)-f(\tilde \zeta,\tilde p)\Big)  +  r(q -\vec{1}/d )G(p),
$$
we have, for any $i=1,\dots,d$,
$$   p_i^*\le p_i+\delta \dm(p)+rG(p)\le p_i + (1-p_i)(\delta  +r) \le 1, $$
$$   p_i^*\geq p_i-\delta \dm(p)-rG(p)\geq p_i (1-\delta  -r) \geq 0. $$

To conclude the proof, let us show that $\sum_{i=1}^dp_i^*=1$. Since $f(\zeta,p)$ and
$f(\tilde \zeta,\tilde p)$ are vectors whose  components are  all equal to zero but one which
is equal to one, we have $$\sum_{i=1}^d \Big[f_i(\zeta,p)-f_i(\tilde \zeta,\tilde p)\Big]=0,$$
and, since $q\in \sx$, we have
$$
\sum_{i=1}^d (q_i -1/d)=0.
$$
 Then
\begin{eqnarray*}
  \sum_{i=1}^dp_i^*
 =  \sum_{i=1}^dp_i+\delta \dm f(\zeta)^TA f(\tilde \zeta) \Big(\sum_{i=1}^df_i(\zeta)-f_i(\tilde \zeta)\Big)+
      rG(p)\sum_{i=1}^d (q_i -1/d)
=  1.
\end{eqnarray*}
\end{proof}

\section{A Boltzmann-like equation and its grazing limit.  }

We now consider an infinite population of agents interacting through the game defined by
the matrix $A$ and updating after each interaction their strategies according to the rule
\eqref{interaccion}. We denote $u_t$ the distribution of agents in the simplex of mixed
strategies at time $t$. Notice that $u_t$ is thus a probability measure on $\sx$. Intuitively, if
this probability is regular enough, we can consider that $u_t$ is its density, and  $u_t(p)dp$
is roughly the proportion of agents whose strategy belongs to a neighborhood of volume
$dp$ around $p$.

\subsection{Boltzmann-type equation}
Let us find an equation describing the time evolution of $u_t$. However, since $u_t$ is a
measure, we only hope to find an equation in weak form, that is, for each observable
$\int_\sx \varphi(p)\, du_t(p)$, with $\varphi\in C(\sx,\R)$. Observe that this integral is the
mean value at time $t$ of some macroscopic quantity. For instance if $\varphi\equiv 1$ then
$\int_\sx \varphi(p)\, du_t(p)=u_t(\sx)$ is the total mass of $u_t$, which should be
conserved. If $\varphi(p)=p$ then $\int_\sx \varphi(p)\, du_t(p)=\int_\sx p\,du_t(p)$ is the
mean strategy in the population. We will see later that it is strongly related to the replicator
equation \eqref{ReplicatorSystem}.

We can also think of $u_t$ as the law of the stochastic process $P_t$ giving the strategy of
an arbitrary agent. Then $\int_\sx \varphi(p)\, du_t(p)=\E[\phi(P_t)]$ is the expected value of
$\phi$.

We assume that interactions take place following a unit rate Poisson process. Notice that if
the Poisson process has constant rate, we can always assume  that the rate is one up to a
linear time re-scaling. Then it is standard to show that
\begin{equation}\label{boltzmannconbeta}
  \frac{d}{dt} \int_\sx \varphi(p) \,du_t (p) =
\int_{ \sx^2} \E[\varphi(p^*)-\varphi(p)] \,   du_t(p)du_t(\tilde p),
\end{equation}
see the book \cite{pareschi2013interacting} for more details.

The following result states the existence and uniqueness of a solutions to this equation

\begin{theorem}\label{existenciaboltzmann}
For any initial condition  $u_0\in \mathcal{P}(\sx)$ there exists a unique
$u\in C([0,\infty),\mathcal{P}(\sx)) \cap C^1((0,\infty),\mathcal{M}(\sx)) $  satisfying
$$
\int_\sx \varphi(p) \, du_t(p) = \int_\sx \varphi(p) \, du_0(p)
+\int_0^t \int_{ \sx^2} \E[\varphi(p^*)-\varphi(p)] \,   du_s(p)du_s(\tilde p)ds
$$
for any test-function $\varphi\in C(\sx)$.
\end{theorem}

\noindent The proof of this result is classical and mainly based on Banach fixed-point
theorem. It can also be proved viewing \eqref{boltzmannconbeta} as an ordinary differential
equation in a Banach space following Bressan's insight \cite{bressan2005notes} (see also
\cite{alonso2016cauchy}). For the reader convenience we provide the main steps of the proof
in Appendix.

\subsection{Grazing limit}

We fix an initial condition  $u_0\in \mathcal{P}(\sx)$ and denote $u^\delta$ the solution of \eqref{boltzmannconbeta}
given by Theorem \ref{existenciaboltzmann} corresponding to interaction rules \eqref{interaccion}.
Notice that when  $\delta \simeq 0$, $|p^*-p|\simeq 0$ so that
 $\varphi(p^*)-\varphi(p)\simeq (p^*-p)\varphi'(p) + \frac12 (p^*-p)^2\varphi''(p)$.
Taking its expected value  we thus obtain
$$ \E[\varphi(p^*)-\varphi(p)]
\simeq \E[p^*-p]\varphi'(p) + \frac12 \E[(p^*-p)^2]\varphi''(p). $$
Using the rules \eqref{interaccion} to calculate the expectation and considering the new time scale $\tau = \delta t$,
we obtain that we can  approximate \eqref{boltzmannconbeta}
when $\delta \simeq 0$ by a local equation of the form
\begin{equation}\label{LimitEq}
 \p_\tau v+\text{div}(\F[v]\,v)=\frac{\lambda }{2} \sum_{i,j=1}^dQ_{ij}\partial_{ij}(G^2\, v).
\end{equation}
Here  $\lambda := r^2 /\delta$, the vector-field $\F[v]$ has components
\begin{equation}\label{defcampo}
\begin{split}
  \F_i[v_\tau ](p ) & = \sum_{k=1}^d \dm(p) a_{ik}(p_i\pp_k(\tau ) +p_k\pp_i(\tau )) \\
                   & = \dm(p) (p_i e_i^TA \pp(\tau) + \pp_i(\tau)e_i^T A p),
\end{split}
\end{equation}
being
$$\pp(\tau )=\int_\sx p\,dv_\tau (p)$$
the mean-strategy at time $\tau $,
and the diffusion coefficient $Q_{ij}$ is the covariance of the distribution of $\theta$, namely
\begin{equation}\label{defQ}
Q_{ij}:= \int_\sx (q_i -1/d )(q_j -1/d )\,d\theta(q).
\end{equation}

We say that $v\in C([0, \infty), \mathcal{P}(\sx))$ is a solution to equation \eqref{LimitEq} if
 \begin{equation}\label{DefWeakSol}
 \begin{split}
   & \int_\sx \varphi(p)\,dv_t(p) -  \int_\sx \varphi(p)\,du_0\\
   & \quad = \int_0^t \int_{\sx}\nabla\varphi(p)\cdot \mathcal{F}(p,s)\,dv_s(p)
     +\frac{\lambda}{2}\sum_{i,j=1}^d Q_{ij} \int_{\sx } \p_{ij}\varphi(p)  G^2(p)  \, dv_s(p) ds
 \end{split}
 \end{equation}
for any $\varphi\in C^2(\sx)$.

The above procedure is relatively well-known in the literature as grazing-limit. It has been
introduced in the context of socio and econophysics modelling by Toscani
\cite{toscani2006kinetic}, see also \cite{pareschi2013interacting}. It can be rigourously
justified to obtain the following Theorem:

\begin{theorem}\label{grazing}
Given an initial condition $u_0\in\mathcal{P}(\sx)$, let $u^\delta$  be the solution to
equation \eqref{boltzmannconbeta} given by Theorem \ref{existenciaboltzmann}
corresponding to interaction rule \eqref{interaccion}.

Assume that, as $\delta,r\to 0$, we have $r^2 /\delta\to\lambda$. Let $\tau=\delta t$ and
define $u^\delta_\tau:=u^\delta_t$. Then there exists  $v\in C([0, \infty), \mathcal{P}(\sx))$
such that, as $\delta\to 0$ up to a subsequence, $u_\delta \to v$ in
$C([0,T],\mathcal{P}(\sx))$ for any $T>0$. Moreover, $v$ is a weak solution to equation
\eqref{LimitEq} in the sense of \eqref{DefWeakSol}.

Eventually, if  $r^2/\delta^\alpha \to \lambda >0$ for some $\alpha\in(0,1)$, then re-scaling
time considering
 $\tau=\delta^\alpha t$, we obtain that  $u_\delta\to v $ as before with $v$ solution to
\begin{equation*}
 \frac{d}{d\tau}v=\frac{\lambda }{2} \sum_{i,j=1}^dQ_{ij}\partial_{ij}(G^2\, v) .
\end{equation*}
On the other hand, if  $r^2/\delta  \to 0$, then re-scaling time we obtain that  $u_\delta\to v
$  with $v$ solution to
\begin{equation}\label{LimitEqq}
\p_\tau v+\text{div}(\F[v]\,v)=0.
\end{equation}
\end{theorem}

In the rest of the paper we will focus in this last case,  corresponding to the pure transport
equation  \eqref{LimitEqq}. Observe that this is a first order, nonlocal, mean field equation,
and following a classic strategy going back at least to Neunzert and Wik
\cite{neunzert1974approximation}, see also
\cite{braun1977vlasov,canizo2011well,dobrushin1979vlasov}, we can prove directly the
well-posedness of equation \eqref{LimitEqq}.

 Given $v\in C([0,+\infty],P(\sx))$ we denote $T_t^v$ the flow of the
vector field $\mathcal{F}[v(t)](x)$ namely
$$ \frac{d}{dt} T_t^v(x) = \mathcal{F}[v(t)](T_t^v(x)), \qquad T_{t=0}^v(x)=x. $$
It can be proved (see Appendix) that $T_t^v(x)\in\sx$ for any $t\ge 0$.
The result is the following:

\begin{theorem}\label{teotransporte}
For any initial condition $u_0\in P(\sx)$, equation \eqref{LimitEqq} has a unique solution $u$
in $C([0,+\infty],P(\sx))$. This solution  satisfies
 $u(t)=T_t^v\sharp u_0$ for any $t\ge 0$.

Moreover, the solutions depend continuously on the initial conditions. Indeed, there exists a
continuous function  $r:[0,+\infty)\to [0,+\infty)$ with $r(0)=1$ such that for any pair of
solutions $v^{(1)}$ and $v^{(2)} $ to equation \eqref{LimitEqq} there holds
\begin{equation*}
	W_1(v^{(1)}(t),v^{(2)}(t))\leq r(t) W_1(v^{(1)}(0),v^{(2)}(0)).
\end{equation*}
\end{theorem}

The proofs of Theorems \ref{grazing} and \ref{teotransporte} can be found in the Appendix.

\section{Relationships between the mean-field equation, the replicator equations,  and the game. }

In this section we study the relationships between:
\begin{itemize}
	\item solutions $v\in C([0,+\infty],\Prob(\Delta))$  to
	the mean-field equation equation \eqref{LimitEqq}, or in weak form,
	\begin{equation}\label{TransportEqu}
	\frac{d}{dt} \int_\sx \phi\,dv_t = \int_\sx \F[v]\nabla\phi\,dv_t \qquad \text{for any $\phi\in C(\sx)$,}
	\end{equation}
	where  the vector-field $\F[v]$ is given by \eqref{defcampo},
	\item solutions to the replicator equations \eqref{ReplicatorSystem} $$	
\frac{d}{dt}p_i=p_i((Ap)_i-p^TAp) \qquad i=1,\ldots,d,
 $$
	\item Nash equilibria of the symmetric zero-sum game  with pay-off matrix $A$.
\end{itemize}

We first relate the stationary solution to the mean-field equation  \eqref{LimitEqq} of the
form $\delta_q$ with $q$ an interior point and the Nash equilibria of the game. Indeed, we
will prove that  $\delta_q$ is a stationary solution if and only if $q$ is a Nash equilibrium.

We then show that the mean strategy of the population satisfies the replicator equations.
Finally, we will study the case of a two-strategies game, where we can precisely describe
the asymptotic behavior of $v_t$, and then we show that, for generalizations of the
classical rock-paper-scissor where the trajectories of the replicator equations are closed
orbits, $v_t$ is also periodic.

\subsection{Nash equilibria and stationary solutions. }

Given some probability measure $v\in\Prob(\sx)$, we will slightly abuse notation considering  $v$ as the time-continuous function
from $[0,+\infty)$ to $\sx$ constantly equal to $v$.

\begin{defi}\label{definicionestransporte}
We say that $v\in \Prob(\sx)$ is an equilibrium or stationary solution of the transport
equation \eqref{TransportEqu} if it is a
 solution of  \eqref{TransportEqu}, that is, if
\begin{equation}\label{DeltaEquilibrium}
\int_\sx \F[v](p)\nabla\phi(p)\,dv(p) = 0 \qquad \text{for all $\phi\in C(\sx)$.}
\end{equation}
\end{defi}

We will mainly be interested in the case of equilibrium of the form $v=\delta_q$ for some
interior point $q\in\sx$.

\begin{theorem}\label{teoequilibrio}
 Let $q$ be an interior point of $\sx$. The following statements are equivalent:
\begin{enumerate}
  \item\label{equilibrioreplicador} $q$ is an equilibrium of the replicator equations \eqref{ReplicatorSystem},
  \item\label{equilibriomodelo} $\delta_q$ is an equilibrium of equation
      \eqref{TransportEqu}  in the sense of definition \ref{definicionestransporte},
  \item\label{equilibrioautovector} $q$ belongs to the null space of the matrix  $A$,
  \item\label{equilibrionash} $q$ is a Nash equilibrium of the symmetric zero-sum game with pay-off matrix $A$.
\end{enumerate}
\end{theorem}

\begin{proof}
Let us first rewrite condition 	\eqref{DeltaEquilibrium} for $\delta_q$ to be an equilibrium of
equation \eqref{TransportEqu}. First notice that, for $v=\delta_q$, the mean strategy is
obviously $q$. Then from definition  \eqref{defcampo} of the vector-field $\F_i[\delta_q]$ we
have for any $p\in \sx$ and any time $t\ge 0$ that
$$ \F_i[\delta_q](p,t)= \sum_{k=1}^d \dm(p) a_{ik}(p_i q_k +p_kq_i). $$
In particular with $p=q$ we have
$$ \F_i[\delta_q](q,t)= \sum_{k=1}^d \dm(q) a_{ik}(q_i q_k +q_kq_i)
=2\dm(q)q_ie_i^tAq =2\dm(q)q_i(e_i^tAq-q^TAq) $$
where we used that $A$ is antisymmetric so that  $q^TAq=0$.
We can now rewrite condition \eqref{DeltaEquilibrium}.
Notice that
$$ \int_\sx \F[\delta_q]\nabla\phi\,d\delta_q =  \F[\delta_q](q)\nabla\phi(q)
 = 2\dm(q)\sum_{i=1}^d q_i(e_i^tAq-q^TAq) \p_i\phi(q). $$
Recall that $q$ is an interior point of $\sx$ so that $h(q)\neq 0$. We thus  obtain that
$\delta_q$ is an equilibrium of equation \eqref{TransportEqu} if and only if
\begin{equation}\label{DeltaEquilibrium2}
\sum_{i=1}^d q_i(e_i^tAq-q^TAq) \p_i\phi(q) = 0 \qquad \text{for all $\phi\in C(\sx)$.}
\end{equation}

We can now easily prove that statements (1) and (2) are equivalent. Indeed if $q$ is an
equilibrium of the replicator equations then $q_i((Aq)_i-q^TAq)=0$ for any $i=1,\ldots,d$ and
\eqref{DeltaEquilibrium2}  holds. On the other hand if $\delta_q$ is an equilibrium of the
equation \eqref{TransportEqu}, i.e. condition  \eqref{DeltaEquilibrium2} holds for any $\phi\in
C(\sx)$, then taking $\phi(p)=p_i$ we obtain $q_i(e_i^tAq-q^TAq)=0$ for any $i=1,\ldots,d$
so we get that $q$ is an equilibrium of the replicator equation.

We can also prove that (1) and (3) are equivalent. Indeed if (1) holds, then
$q_i(e_i^tAq-q^TAq)=0$ for any $i=1,\ldots,d$. Since $q^TAq=0$ and  $q_i>0$ for any
$i=1,\ldots,d$, being $q$ an interior point, the previous equality can be rewritten as
$e_i^tAq=0$ for any $i=1,\ldots,d$ i.e. $Aq=0$. This prove that (1) implies (3). On the other
hand, if $Aq=0$, then $e_i^tAq=0$ for any $i=1,\ldots,d$ so we obtain that
\eqref{DeltaEquilibrium2} holds.

It remains to show that $Aq=0$ if and only if $q$ is a Nash equilibrium.
Suppose that $Aq=0$. Then for any $p\in \sx$,
$$p^TAq=p\cdot \vec{0}=0=q^TAq. $$
Thus, playing any other strategy than $q$  against $q$ does not increase the pay-off. This
means that $q$ is a Nash equilibrium. Let us now assume that $q$ is a Nash equilibrium and
let us prove that $(Aq)_i=0$ for any $i=1,\ldots,d$. If $(Aq)_i>0$  then $e_i^{T}Aq>0=q^TAq$
contradicting that $q$ is a Nash equilibrium. If $(Aq)_i<0$ then recalling that $A$ is
antisymmetric,
 $$0=q^TAq=\sum_{k=1}^dq_k(Aq)_k=q_i(Aq)_i+\sum_{k\neq i}q_k(Aq)_k. $$
Since $q_k>0$  for any $k=1,\ldots,d$ there must exists some $l\in \{1,\ldots,d\}$ such that
$(Aq)_l>0$ and  this is not possible.

The proof is finished.
\end{proof}

\subsection{Evolution of the mean strategy and the replicator equations.}

The method of characteristic yields that the solution to equation  \eqref{TransportEqu} is
\begin{equation}\label{SolTransportEqu}
v_t = T_t\sharp v_0
\end{equation}
where $T_t$ is the flow of the vector-field $\F[v_t](p)$. This vector-field is the same as the
one in the replicator equations but with the mean-strategy $\pp(t)$ depending on the
distribution $v_t$ of strategies. The next result shows that $\pp(t)$ satisfies (up to a
constant) the replicator equations.

\begin{theorem}\label{dinamicapromedio}
	Consider a solution $v\in C([0,\infty),\Prob(\sx))$ of the transport equation
	\eqref{TransportEqu} staying away from the boundary $\p\sx$ of $\sx$ up to some time $T>0$ i.e.
	\begin{equation}\label{HipDinProm}
	dist\Big(supp(v_t),\p\sx\Big)\ge c^{1/d} \qquad 0\le t\le T,
	\end{equation}
	where $c$ is defined in \eqref{defdelta}.
	Then the mean strategy $\bar p(t) = \int_\sx p\,dv_t(p)$ is a solution of the replicator equation:
	\begin{eqnarray}\label{EDOMeanStrategy}
	\frac{d}{dt} \pp_i(t) =  2c \pp_i(t) e_i^T A\pp(t) \qquad i=1,\ldots,d.
	\end{eqnarray}	
\end{theorem}

\noindent Notice that \eqref{EDOMeanStrategy} is not exactly the replicator systems due to the
constant $2c$, but becomes so after the time-scale change $\tau=2ct$.

\begin{proof}
	Let $\T_{s,t}$ be the flow of the vector-field $\F[v](t,x)$ i.e.
	$$ \frac{d}{dt}\T_{s,t}(p) = \F[v](\T_{s,t}(p),t), \qquad \T_{s,s}(p)=p. $$
	We also let $\T_t(p)=\T_{0,t}(p)$ and denote $\T^i_t(p)$, $i=1,\ldots,d$, its components.
	Then $v_t = \T_t\sharp v_0$. It follows that for any $i=1,\ldots,d$,
	\begin{eqnarray*}
		\pp_i(t) = \int_\sx p_i \, dv_t(p) = \int_\sx \T^i_t(p) \, dv_0(p),
	\end{eqnarray*}
	so that
	\begin{eqnarray*}
		\frac{d}{dt} \pp_i
		& = & \int_\sx \frac{d}{dt} \T^i_t(p) \, dv_0(p)
		= \int_\sx  \F_i[v](\T_{s,t}(p),t)\, dv_0(p) \\
		& = & \sum_{k=1}^d \int_\sx  \dm\Big(\T_t(q)\Big) a_{ik}\Big(\T^i_t(q)\pp_k(t) +\T^k_t(q)\pp_i(t)\Big)    \, dv_0(q) \\
		& = & \sum_{k=1}^d  \int_\sx  \dm(p) a_{ik} (p_i\pp_k(t) +p_k\pp_i(t) )    \, dv_t(p),
	\end{eqnarray*}	
	where we used once again that $v_t = \T_t\sharp v_0$.
	According to assumption \eqref{HipDinProm} we have for any $p$ in the support of $v_t$ that
	$p_i\ge c^{1/d}$ for $i=1,\ldots,d$, which implies that  $h(p)=c$. Thus,
	\begin{align*}
		\frac{1}{c}\frac{d}{dt} \pp_i(t)
		=& \sum_{k=1}^d  \int_\sx    a_{ik} (p_i\pp_k(t) +p_k\pp_i(t) )    \, dv_t(p)  \\
		= &  2\sum_{k=1}^d  a_{ik} \pp_i(t)\pp_k(t)
	\end{align*}	
	which is \eqref{EDOMeanStrategy}.
\end{proof}

\subsection{Two strategies games.}

Let us consider the case of a symmetric game with two strategies. The pay-off matrix $A$ is then
$$ A = \begin{pmatrix} 0 & b \\ -b & 0 \end{pmatrix} $$
for some $b\in \R$. Notice that if $b>0$ (resp. $b<0$) then the first (respectively, second) strategy
strictly dominates the other.
We thus expect that all agents end up playing the dominating strategy except those initially playing
exclusively the loosing strategy and they cannot move due to the presence of $h$ in the interaction rule.

Let $v_t$ be the solution to the transport equation \eqref{LimitEqq} and $\mu_t$ be the distribution of $p_1$ (i.e. $\mu_t$ is the first marginal of $v_t$.)
This means that $v_t(A\times [0,1])=\mu_t(A) $ for any Borel set $A\subset [0,1]$, which can be rewritten as
$$ \iint \phi(p_1) \,dv_t(p_1,p_2) = \int \phi(p_1)\,d\mu_t(p_1)  $$
for any measurable non-negative function $\phi:[0,1]\to\R$.

\begin{theorem}\label{Thm2strategies}
	Assume that $b>0$ and write the initial condition as
	$$ \mu_0 = (1-a)\delta_0+a\tilde \mu_0$$
	where $a\in [0,1]$  and $\tilde \mu$ is a probablity measure on $[0,1]$ such that $\tilde \mu_0(\{0\})=0$.
	Then
	$$\lim_{t\to +\infty}\mu_t = (1-a)\delta_0 + a\delta_1.$$
	
	In the same way, if $b<0$ and $\mu_0 = (1-a)\delta_1+a\tilde \mu_0$ where $\tilde \mu(\{1\})=0$ and
	$a\in [0,1]$, then $\mu_t\to (1-a)\delta_1+a\delta_0$ as $t\to +\infty$.
\end{theorem}

\begin{proof}	
	Let us assume that $b>0$ (the proof when $b<0$ is completely analogous).
		It follows from \eqref{LimitEqq} that $\mu_t$ satisfies the following equation: for any $\phi\in C^1([0,1])$,
	\begin{equation}\label{Equp1}
	 \frac{1}{b} \frac{d}{dt} \int_0^1 \phi\,d\mu_t
	= \int_0^1	\phi'(p_1) v[\mu_t](p_1) \,d\mu_t(p_1),
	\end{equation}
	where for any probability measure $\mu$ on $[0,1]$ the vector-field $v[\mu]$ is defined as
	$$  v[\mu](p_1) = \underbrace{\min\{p_1(1-p_1),c\}}_{h(p_1)} (p_1+\bar p_1 - 2p_1\bar p_1),$$
where $$ \bar p_1 = \int_0^1 p_1\,d\mu(p_1). $$
	
	Notice first that since $p_1,\bar p_1\in [0,1]$ we have
	\begin{equation}\label{Claim2}
	v[\mu](p_1) \ge h(p_1)(p_1^2+\bar p_1^2 - 2p_1\bar p_1) = h(p_1) (p_1-\bar p_1)^2.
	\end{equation}
	Hence, it follows that $v[\mu]\ge 0$. Another consequence is that
	\begin{equation}\label{Claim}
	v[\mu] = 0 \text{  $\mu$-a.e.} \qquad \Leftrightarrow \qquad \mu=\alpha\delta_0 + (1-\alpha)\delta_1,\,\alpha \in [0,1].
	\end{equation}
	Indeed, if $v[\mu](p)=0$ then $p_1=0,1,\bar p_1$ by \eqref{Claim2}.
	Thus if $v[\mu] = 0$ $\mu$-a.e. then $\mu = \alpha \delta_0+\beta \delta_1 + \gamma\delta_c$ for some $c\in (0,1)$
	and $\alpha,\beta,\gamma\ge 0$, $\alpha+\beta+\gamma=1$. Since $h(c)\neq 0$,
$v[\mu](c)=0$ gives $c=\bar p_1$ by \eqref{Claim2} and then $0=v[\mu](c)=h(c)(c+c-2c^2)$
i.e. $c=0$ or $c=1$ which is an absurd. 	

	Let us recall that $\mu_t=T_t\sharp\mu_0$
where $T_t$ is the flow of $v[\mu_t](p_1)$, and also that $v[\mu_t](p_1)\ge 0$. Thus for
any $x$ in the support of $\mu_0$, $T_t(x)$ is non-decreasing and bounded by 1 and thus
converge to some $T_\infty(x)$. Then $\mu_t\to \mu_\infty:=T_\infty\sharp\mu_0$.

Moroever, $v[\mu_\infty]=0$ $\mu_\infty$-a.e. and thus
 $\mu_\infty=(1-\alpha)\delta_0+\alpha\delta_1$ for some $\alpha\in [0,1]$ by
 \eqref{Claim}.

	To conclude, we have to show that
	\begin{equation}\label{Claim3}
	\mu_\infty(\{0\})=\mu_0(\{0\}).
	\end{equation}

Let us take a smooth non-increasing function $\phi:[0,1]\to [0,1]$
	such that $\phi=1$ in $[0,1/n]$ and $\phi=0$ in $[2/n,1]$, $n\in \mathbb{N}$. Then
	$$ \int_0^1 \phi\,d\mu_t - \int_0^1 \phi\,d\mu_0 = \int_0^t\int_0^1 \phi'(p_1)v[\mu_s](p_1)\,d\mu_s(p_1)ds \le 0 $$
	Letting $t=t_k\to +\infty$ we obtain
	$$ \int_0^1 \phi\,d\mu_\infty \le \int_0^1 \phi\,d\mu_0 $$
	and then  $ \mu_\infty([0,1/n])\le \mu_0([0,2/n])$. Letting $n\to +\infty$ gives $\mu_\infty(\{0\})\le \mu_0(\{0\})$.
	To prove the converse inequality recall that $\mu_t=T_t\sharp\mu_0$ where $T_t$ is the flow of $v[\mu_t](p_1)$.
	For any  $\phi\in C([0,1])$ we thus have
	$$ \int_0^1 \phi\,d\mu_t = \int_0^1 \phi(T_t(p_1))\,d\mu_0(p_1)
	= (1-a)\phi(0) + a \int_0^1 \phi(T_t(p_1))\,d\tilde \mu_0(p_1)  $$
	Letting $t=t_k\to +\infty$ we obtain $ \int_0^1 \phi\,d\mu_\infty\ge (1-a)\phi(0)$ for any nonnegative and continuous function $\phi$.  We deduce that $\mu_\infty(\{0\})\ge 1-a$.
	This proves \eqref{Claim3}.

We conclude that $\mu_t \to (1-a)\delta_0 + a\delta_1$, and this finishes the proof.
	
\end{proof}

Given an initial condition $\mu_0$, the distribution $\mu_t$ of $p_1$ is the unique solution (see \eqref{Equp1})  of
$$ \frac{1}{b} \p_t\mu_t + \p_{p_1}\Big( v[\mu_t](p_1)   \Big) = 0.$$

In particular if $\mu_0$ is a convex combination of Dirac masses like e.g. $\mu_0 = \frac1N
\sum_{i=1}^N \delta_{p_1^i(0)}$ with $p_1^1(0),\ldots,p_1^N(0)\in [0,1]$, then $\mu_t =
\frac1N \sum_{i=1}^N \delta_{p_1^i(t)}$ where $p_1^1(t),\ldots,p_1^N(t)$ are the solutions
of the system
\begin{equation} \label{Syst2strategies}
\begin{split}
 \frac{d}{dt} p_1^i(t)
 & =   v[\mu](p_1^i(t)) \\
 & =  \min\{p_1^i(1-p_1^i),c\} (p_1^i(t)+\bar p_1 - 2p_1^i(t)\bar p_1)
 \qquad i=1,\ldots,N,
\end{split}
\end{equation}
where
$$ \bar p_1 =  \frac1N \sum_{i=1}^N  p_1^i(t). $$

We solved numerically the system \eqref{Syst2strategies} in the time interval $[0,T]$ using a Runge-Kutta scheme of order 4 with step size $h=0.1$ with the following parameters values:
\begin{equation}\label{2strategies_Param}
 b=1, \qquad c=0.1, \qquad  N=1000, \qquad  T=400,
\end{equation}
and taking as initial condition
\begin{equation}\label{2strategies_CondIni}
\begin{split}
&  p_1^1(0)=...=p_1^{300}(0)=0,  \\
& \text{$p_1^k(0)$, $k=301,\ldots,N$,  uniformly and independently distributed in $[0,0.3]$. }
\end{split}
\end{equation}
We show in figure \ref{Fig_2strategies} the resulting evolution of the distributions of the
$p_1^k(t)$, $k=1,\ldots,N$. We can see that $p_1^1(t)=...=p_1^{300}(t)=0$ for any $t$,
resulting in the Dirac mass $\frac{3}{10}\delta_0$. The others $p_1^k (t)$,
$k=301,\ldots,N$, are moving to the right until reaching 1 thus building up progressively the
Dirac mass $\frac{7}{10}\delta_1$ in complete agreement with Theorem
\ref{Thm2strategies}.

\begin{figure}\label{Fig_2strategies}	
	\centering
	\includegraphics[width=\textwidth]{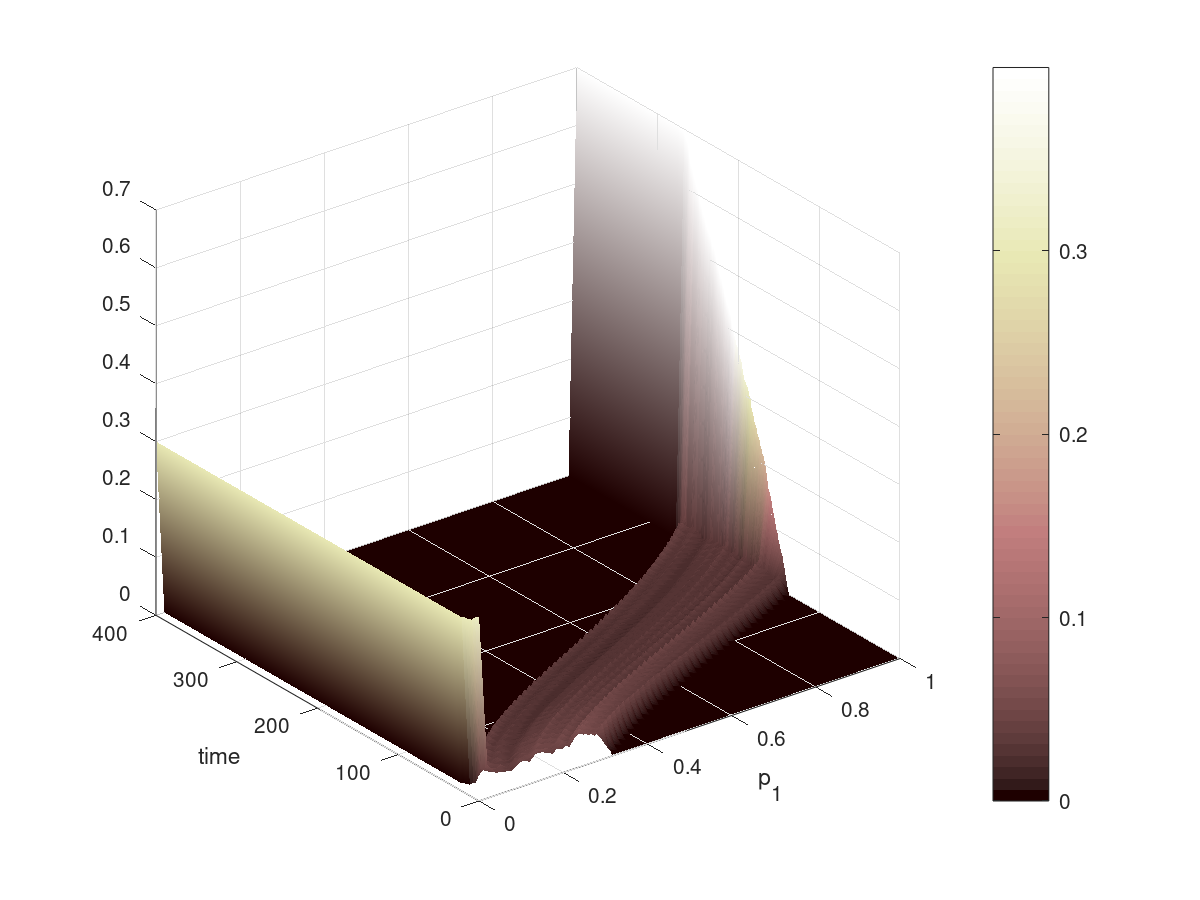}
	\caption{Time evolution of the distribution of $p^1_1,\ldots,p^N_1$ solutions of \eqref{Syst2strategies}
		with parameter values \eqref{2strategies_Param} and initial condition \eqref{2strategies_CondIni}.}
\end{figure}

\subsection{Periodic solutions for Paper-Rock-Scissor like games.}

We now consider the case of a game for which the solutions of the replicator equations
are periodic orbits. We have in mind the classic Rock-Paper-Scissor game whose pay-off matrix is
$$ A = \begin{pmatrix} 0 & 1 & -1 \\ -1 & 0 & 1 \\ 1 & -1 & 0 \end{pmatrix} $$
In this game strategies dominate each other in a cyclic way as $1\to 2\to 3\to 1$.
This can be generalized to a game with an odd number $d$ of strategies considering the pay-off matrix
\begin{equation}\label{PayOffPPT}
 A:= \begin{pmatrix}
	0       & a_1 &  a_2 & \cdots & \cdots & a_{d-1} \\
	a_{d-1} & 0   & a_1 &  a_2 & \cdots & a_{d-2} \\
	a_{d-2} & a_{d-1} & 0 &  a_1 & \cdots & a_{d-3} \\
	\cdots & \cdots & \cdots & \cdots & \cdots & \cdots \\
	a_1 &  a_2 & \cdots & \cdots & a_{d-1} & 0
\end{pmatrix}
\end{equation}
with $a_k=(-1)^{k-1}$, $k=1,\ldots,d-1$.

It is known that interior trajectories for the replicator equations for these games are closed
periodic orbits enclosing the unique interior  Nash equilibrium $N:=\frac{1}{d}\vec{1}$ (see
e.g. \cite{sandholm2010population}). In particular
\begin{equation}\label{Value}
N^T A N = 0.
\end{equation}
Moreover (see Theorem  7.6.4 in \cite{hofbauer1998evolutionary}) if $p(t)$ is such a
trajectory then its temporal mean converges to $N$:
$$ lim_{t\to +\infty} \frac{1}{t}\int_0^t p(s)\,ds = N.  $$
It follows that
\begin{equation} \label{TempMean}
 \frac{1}{T}\int_0^T p(s)\,ds = N
\end{equation}
where $T$ is the period of the trajectory $x(t)$.

Our next result states that the solutions of the mean-field equation equation \eqref{LimitEqq} lying in $\{h=c\}$ are also periodic. %As a by-product of the proof we also obtain the explicit value of the period of the solution to the replicator equation, a result probably well-known but that we could not find in the litterature.

\begin{theorem}
Consider a solution $v_t$  to equation \eqref{LimitEqq} with initial condition $v_0$ such that
$v_t$ is supported in $\{h=c\}$ for any $t\ge 0$.
If the initial mean strategy is different from $N$ then there exists $T>0$ such that
 $v_{t+T}=v_t$ for any $t\ge 0$.
\end{theorem}

\begin{proof}

Consider a solution $v_t$  to equation \eqref{LimitEqq} with initial condition $v_0$.
Then $v_t =\T_t\sharp v_0$ where $\T_t$ is the flow of $\F[v_t](p)$.
Thus to prove that $v_t$ is periodic, it is enough to prove that all the trayectories $t\to\T_t(p)$,
$p\in \text{supp}(v_0)$, are periodic with the same period.
Let $p(t)=\T_t(p)$ be such a trajectory.
Since $h(p(t))=c$ for any $t$,
\begin{eqnarray*}
 \frac{d}{dt} p(t)
 & = & \F[v_t](p(t)) \\
 & = & c(B(t)+C(t)A)p(t)
\end{eqnarray*}
where
$$ B(t)=diag((A m(t))_1,\ldots,(A m(t))_d), \qquad C(t)= diag(m_1(t),\ldots,m_d(t))  $$
$$ m(t)=\int_\Delta p\,dv_t(p). $$
Thus
$$ p(t) = exp\Big(c\int_0^t B(s)+C(s)A\,ds\Big)p(0). $$

According to Theorem \ref{dinamicapromedio},
$m$ is a solution to the replicator equations for $A$ and thus is periodic.
We denote its period by  $T$.
By \eqref{TempMean} we deduce that
$$\frac{1}{T}\int_0^T m(s)\,ds = N. $$
Thus
$$\frac{1}{T}\int_0^T A m(t)\,ds = A\Big( \frac{1}{T}\int_0^T m(s)\,ds \Big) = AN=0. $$
Recalling that $N=\frac{1}{d}\vec{1}$,
$$ \frac{1}{T} \int_0^T B(s)+C(s)A\,ds =  \frac{1}{T} \int_0^T C(s)\,ds.A = \frac{1}{d} A. $$
We deduce that
$$ p(T) = exp\Big(\frac{cT}{d} A\Big)p(0). $$

The matrix $R:=exp\Big(\frac{cT}{d} A\Big)$ is orthogonal (being $A$ antisymmetric, $exp(
A^t) = exp(-A) = (exp A)^{-1}$) and has determinant $exp\Big(Tr(\frac{cT}{d} A)\Big)=1$.
Thus $R\in SO(d)$.

 Moreover, $RN=(Id+\frac{cT}{d} A + ... )N = N$ so we get that $R$ is a  rotation around the
line $(ON)$. Since this line  is perpendicular to the  plane $\{p_1+ \dots +p_d=1\}$, the
matrix $R$ is a rotation in this plane fixing $N$.

We thus obtain that any trajectory $p(t)$ starting from $p(0)\in \text{supp}(v_0)$ satisfies
$$ p(T)=Rp(0) $$
where $T$ is the period of $m$ and $R$ is a rotation in $\{p_1+ \dots +p_d=1\}$ fixing $N$.
It follows in particular that $v_T=R\sharp v_0$ and $m(T)=Rm(0)$. Since $m(T)=m(0)$ by
definition of $T$, we obtain that $m(0)$ is a another fixed-point of $R$. Thus if $N\neq
m(0)$ then  $R=Id$ so that  $v_T=v_0$. Since the solution to the mean-field equation
\eqref{LimitEqq} with a given intial condition is unique, we deduce that $v_{t+T}=v_t$ for
any $t\ge 0$.

The proof is finished
\end{proof}

\begin{remark}
Let us note that, given a ball $B$ centered at $N$, any trajectory of the replicator equations
starting at some $p\in B\cap \{p_1+ \dots +p_d=1\}$ cannot  reach the boundary of
$\Delta$, since whenever a coordinate of $p(t)$ is equal to zero, it remains zero for every
$t$. So, choosing $B$ sufficiently small, all the trajectories  remain in the set $\{ h=c\}$.

Assuming that $supp(v_0)\subset B$, we get that $supp(v_t)\subset\{ h=c\}$.
\end{remark}

\section{Appendix}

\subsection{Existence of solution to the  Boltzmann-like  equation\label{existenciaboltzmannapendice}:
proof of Theorem \ref{existenciaboltzmann}.}

Given an initial condition $u_0\in \mathcal{P}(\sx)$ we want to prove that there exists a
unique  $u\in C([0,\infty),\mathcal{P}(\sx))$ such that
\begin{equation}\label{teoremaboltzmann10}
  \begin{split}
  \int_\sx \varphi(p) \, du_t(p)=&\int_\sx \varphi(p) \, du_0(p)
  +\int_0^t\int_{\sx^2} \mathbb{E}[\varphi(p^*)-\varphi(p)] \, du_s(p)du_s(\tilde p)ds
\end{split}
\end{equation}
for any  $\varphi\in C(\sx).$

We split the proof into two steps.

\begin{step}
There is a unique $u\in C([0,\infty),M(\sx))$ satisfying \eqref{teoremaboltzmann10}.
\end{step}

Recall that $M(\sx)$ denotes the space of finite Borel measures on $\sx$ that we endow,
unless otherwise stated, with the total-variation norm.

\begin{proof}   Given $u,v\in M(\sx)$ we define a finite measure $ Q(u,v)$ on $\sx$ by
\begin{align*}
  \langle Q(u,v),\varphi\rangle
  :=&\, \frac12\int_{  \sx^3} \mathbb{E}[\varphi(p^*)-\varphi(p)]\,  du(p)dv(\tilde p)
  +\,\frac12 \int_{  \sx^3} \mathbb{E}[\varphi(p^*)-\varphi(p)]\,  dv(p)du(\tilde p)
\end{align*}
for  $\varphi\in C(\sx)$. We also let $Q(u):=Q(u,u)$.
For $u\in C([0,+\infty],M(\sx))$ we then define a map  $J(u):[0,+\infty)\to M(\sx)$
by
$$ J(u)_t:=u_0+\int_0^tQ(u_s)\,ds, \qquad t\ge 0, $$
that is,
$$ (J(u)_t,\varphi):=(u_0,\varphi)+\int_0^t (Q(u_s),\varphi)\,ds \qquad
\text{for any $\varphi\in C(\sx)$.} $$
We thus look for a fixed point of $J$ in $C([0,+\infty],M(\sx))$.
We will apply Banach fixed-point theorem to $J$ in the complete metric space
\begin{equation*}
	\mathcal{A}:=\{u\in C([0,T],\mathcal{M}(\sx)):\, u(0)=u_0 \text{ and } \max_{0\leq s\leq T}\|u_s\|_{TV} \leq 2 \}
\end{equation*}
where $T\in (0,1/8)$.

Let us verify that $J(\mathcal{A})\subseteq \mathcal{A}$.
First notice that
\begin{equation}\label{cotaQ}
\|Q(u,v)\|_{TV} \leq 2\|u\|_{TV}\|v\|_{TV}.
\end{equation}
Moreover, $Q$ is bilinear so that $Q(u)-Q(v)=Q(u+v,u-v)$ and then
\begin{equation}\label{contraccion}
\|Q(u)-Q(v)\|_{TV}\leq 2  \|u+v\|\|u-v\|_{TV}.
\end{equation}
Now for any $u,v\in \mathcal{A}$, and any $t\in [0,T]$,
\begin{align*}
  \|J(u)_t)\|_{TV} &
  \leq \|u_0\|_{TV}+\int_0^t\|Q(v_s)\|_{TV}\, ds \\
  & \leq 1+2T\max_{0\leq s\leq T} \|v_s\|_{TV}^2 \\
 &  \leq 1+8T
  \leq 2.
\end{align*}
Moreover, for $0\le s\le t\le T$,
$$ \|J(u)_t-J(u)_s\|_{TV}
\le \int_s^t \|Q(u_\tau)\|_{TV}\,d\tau
\le 2\int_s^t \|u_\tau\|^2_{TV}\,d\tau
\le 8|t-s|,
$$
and we deduce the continuity of $J(u)_t$ in $t$.
Thus $J(\mathcal{A})\subset \mathcal{A}$.

Now, for any $u,v\in \mathcal{A}$,  using \eqref{contraccion},
\begin{align*}
 \|J(u)_t-J(v)_t\| & \leq \int_0^t\|Q(u_s)-Q(v_s)\|\, ds \\
  & \leq  \int_0^t2\| u_s+v_s \|\| u_s-v_s \|\, ds\\
  & \leq   8T\| u-v \|,
\end{align*}
so that
$$ \|J(u)-J(v)\|\leq  8T\| u-v \|. $$
Thus choosing $T<1/8$, we deduce that $J$ is a strict contraction from $\mathcal{A}$ to
$\mathcal{A}$ and therefore, has a unique fixed-point. Repeating the argument on $[T,2T],
\ldots$, we obtain a unique $u\in C([0,+\infty],M(\sx))$ satisfying
\eqref{teoremaboltzmann10}.
\end{proof}

Notice that it is not a priori obvious that $J(u)_t\ge 0$ if $u_t\ge 0$, $t\ge 0$. We verify
that $u_t\ge 0$ in an indirect way in the next Step following ideas from
\cite{cercignani2013mathematical}.
\begin{step}
$u_t$ is a probability measure on $\sx$ for any $t\ge 0$ where $u$ is given by the previous Step.
\end{step}

\begin{proof}
We first verify that $u_t$ is a non-negative measure for any $t\ge 0$ i.e.
$$(u_t,\varphi)\geq 0 \qquad \text{for any $\varphi\in C(\sx)$, $\varphi\geq 0$.} $$
Given $u,v\in M(\sx)$ we define the measures
$$ Q_+(u,v):=
 \frac12\int_{\sx^3}\mathbb{E}[\varphi(p^*)] \,(du(p)dv(\tilde p)+dv(p)du(\tilde p))    $$
and $Q_+(v):=Q_+(v,v)$.
Notice that $Q_+(u,v)$ is non-negative if both $u$ and $v$ are non-negative. Moreover,
\begin{equation}\label{relacionQQmas}
 Q(u)=Q_+(u)-u
\end{equation}
and
\begin{equation}\label{desigualdadparainduccion}
  Q_+(u)\geq Q_+(v)\geq 0  \qquad \text{if $u\geq v\geq 0$}
\end{equation}
since $Q_+(u)-Q_+(v)=Q_+(u+v,u-v)\geq 0$.

The idea of the proof consists in finding $v\in C([0,\infty),\mathcal{P}(\sx))$ (continuity with
respect to the total variation norm) such that  \begin{equation}\label{vbuscada}
  v_t=e^{-t}u_0+\int_0^te^{s-t}Q_+(v_s)\, ds.
\end{equation}
Indeed in that case using \eqref{relacionQQmas},
\begin{align*}
  \frac{d}{dt}v_t
  &=-e^{-t}u_0-\int_0^te^{s-t}Q_+(v_s)\, ds+Q_+(v_t) \\
 & =Q_+(v_t)-v_t\\
  &=Q(v_t).
\end{align*}
Thus,  $v_t$ verifies \eqref{teoremaboltzmann10} so that $v=u$ since $u$ is the unique solution of \eqref{teoremaboltzmann10}.

To obtain $v$ satisfying \eqref{vbuscada} we consider the sequence
$v^{(n)}\in C([0,+\infty),P(\sx))$, $n\in\mathbb{N}$, defined by
$v^{(0)}:=0$ and
  $$v_t^{(n)}:=e^{-t}u_0+\int_0^te^{s-t}Q_+(v^{(n-1)}_s)\, ds.$$
Recalling that $u_0\ge 0$ and using \eqref{desigualdadparainduccion} it is easily seen that
$v^{(n)}_t\geq v^{(n-1)}_t\ge 0$.
Also notice that
$$ (v^{(1)}_t,1) = e^{-t} + \int_0^te^{s-t}(Q_+(u_0),1)\, ds
= (Q_+(u_0),1) = (Q(u_0),1)+(u_0,1)=1
$$
where we used \eqref{relacionQQmas} and the fact that $(Q(u),1)=0$ for any $u\in M(\sx)$.
Notice that the function  $(v^{(n)}_t,1)$ satisfies the ordinary differential equation
$$ \begin{cases}
\frac{d}{dt} (v^{(n)}_t,1)
= (Q_+(v_t^{(n-1)}),1) -  (v^{(n)}_t,1)
= (v_t^{(n-1)},1) -  (v^{(n)}_t,1),  \\
(v^{(n)}_{t=0},1) = (u_0,1)=1.
\end{cases}$$
We can then prove by induction that $(v^{(n)}_t,1)=1$ for any $n$ and $t$. Thus for any $t\ge
0$, $(v_t^{(n)})_n$ is a non-decreasing sequence of probability measures on $\sx$. We can
then define a probability measure $v_t$ on $\sx$ by
$$ (v_t,\phi):=\lim_{n\to +\infty} (v^{(n)}_t,\phi) \qquad \phi\in C(\sx). $$
In fact the convergence of $v^{(n)}$ to $v$ is uniform in $t\in [0,T]$ for any $T>0$, and thus
$v$ is continuous in $t$. This follows from the Arzela-Ascoli thorem. Indeed, since
$\|v_t^{(n)}\|_{TV}=1$, we only need to prove that the sequence $(v_t^{(n)})_n$ is uniformly
equicontinuous. We have
\begin{eqnarray*}
\|v_{t+h}^{(n)}-v_t^{(n)}\|_{TV}
& \le & |e^{t+h}-e^t|\|u_0\|_{TV}
+ \int_t^{t+h} e^{s-(t+h)} \|Q_+(v_s^{(n-1)})\|_{TV}\,ds  \\
&& + \int_0^t |e^{s-(t+h)}-e^{s-t}| \|Q_+(v_s^{(n-1)})\|_{TV}\,ds.
\end{eqnarray*}
In view of  \eqref{cotaQ} and recalling that $v_s^{(n-1)}\in \Prob(\sx)$ we have
$\|Q_+(v_s^{(n-1)})\|_{TV}\le \|Q(v_s^{(n-1)})\|_{TV} + \|v_s^{(n-1)}\|_{TV}\le 3$.
The uniform equi-continuity follows easily.

This ends the proof.
\end{proof}

To conclude the proof of Theorem \ref{existenciaboltzmann}, we verify that $u\in
C^1((0,+\infty),\mathcal{M}(\sx))$ with $\p_tu_t=Q(u_t)$. Indeed, recalling that
$u_t=J(u)_t$, we have
$$ \frac{u_{t+h}-u_t}{h}-Q(u_t)
= \frac1h \int_t^{t+h}Q(u_s)\,ds-Q(u_t)
= \frac1h\int_t^{t+h} Q(u_s)-Q(u_t)\,ds. $$
Using \eqref{contraccion} together with $\|u_t\|_{TV}=1$ we obtain
\begin{align*}
\Big{\|}\frac{u_{t+h}-u_t}{h}-Q(u_t)\Big\|_{TV}
\leq  & \frac2h  \int_t^{t+h}  \|u_s+u_t\|_{TV}\|u_s-u_t\|_{TV}\,ds\\
	\leq  & \frac4h  \int_t^{t+h}   \|u_s-u_t\|_{TV}\,ds
\end{align*}
which goes to 0 as $h\to 0$ by the Dominated Convergence Theorem since $u_s$ is
continuous in $s$ for the total variation norm.

\subsection{Grazing limit: proof of Theorem \ref{grazing}.}

The  proof consists in two main steps. First approzimating the difference
 $\varphi(p^*)-\varphi(p)$ in the Boltzmann-like equation \eqref{boltzmannconbeta}
by a second order Taylor expansion, we obtain that $u_\delta $ is an approximate s solution of
\eqref{LimitEq} in the sense of \eqref{DefWeakSol}. Then we apply Arzela-Ascoli Theorem to
deduce that a subsequence of the $u_\delta$ converges to a solution of \eqref{DefWeakSol}.

Before beginning the proof we need the following lemma which gives the expected value of
of $f(\zeta,p)^TA f(\tilde \zeta,\tilde p)\Big(f_i(\zeta,p)-f_i(\tilde \zeta,\tilde p)\Big)$ where
the random vector $ f(\zeta; p) = ( f_1(\zeta; p),\ldots, f_d(\zeta; p))$ is defined in
\eqref{definicionf}.

\begin{lem} For any  $i,j=1,\dots,d$, any  $p,\tilde p\in\sx$ and any
 independent random variables $\zeta,\tilde\zeta$ uniformly distributed in $[0,1]$, there hold
	\begin{equation}\label{Acumulada1}
	\mathbb{E}\Big[ f(\zeta,p)^TA f(\tilde \zeta,\tilde p)\Big(f_i(\zeta,p)-f_i(\tilde \zeta,\tilde p)\Big) \Big]
	=\sum_{k=1}^d  a_{ik}(p_i  \tilde p_{k}  +   p_k\tilde p_{i} ).
	\end{equation}
\end{lem}

\begin{proof}
	Let us denote $f_i=f_i(\zeta,p)$ and $\tilde f_i=f_i(\tilde \zeta,\tilde p)$.
	The proof is based on the following two properties of the $f_i$:
	$$ f_if_j=\delta_{ij}\qquad \text{and} \qquad f_i^2=f_i,$$
	which follows from the definition of $f(\zeta,p)$.
	We write
	\begin{equation*}\label{Grazing1}
	\begin{split}
	f(\zeta,p)^TA f(\tilde \zeta,\tilde p)\Big(f_i(\zeta,p)-f_i(\tilde \zeta,\tilde p)\Big)
	& =  \sum_{m,n=1}^d a_{mn}f_mf_n\Big(f_i-f_i\Big) \\
	& = \sum_{n=1}^d a_{in}f_i f_n-  \sum_{m=1}^d a_{mi}f_mf_i.
	\end{split}
	\end{equation*}
	Since $\zeta$ and $\tilde\zeta$ are independent, so are $f_i(\zeta,p)$ and
	$f_j(\tilde \zeta,\tilde p)$ for any $i,j=1,\ldots,d$.
	Moreover $\mathbb{E}[f_i(\zeta,p)]=p_i$ since $f_i(\zeta,p)=1$ with probability $p_i$ and
	$f_i(\zeta,p)=0$ with probability $1-p_i$.
	Taking the expectation in \eqref{Grazing1} we thus obtain
	$$ \mathbb{E}\Big[ f(\zeta,p)^TA f(\tilde \zeta,\tilde p)\Big(f_i(\zeta,p)-f_i(\tilde \zeta,\tilde p)\Big) \Big]
	=\sum_{n=1}^d a_{in}p_i \tilde p_m -  \sum_{m=1}^d a_{mi}p_m \tilde p_i. $$
	We deduce \eqref{Acumulada1} recalling that $a_{ij}=-a_{ji}$ being $A$ antisymmetric.

The proof is finished.
\end{proof}

We are now in position to prove Theorem \ref{grazing}.
We split the proof into several Steps.
The first one states that $u_\delta $ is an approximate s solution of  \eqref{LimitEq} in the sense of \eqref{DefWeakSol}.

\begin{step}
For any $\phi\in C^3(\sx)$,
\begin{equation}\label{Step1}
\begin{split}
& \int_\sx  \varphi(p)\,du_\delta (p,\tau)
- \int_\sx  \varphi(p)\, du_{\delta }(p,0)\\
& =  \int_0^t \int_{\sx} \nabla\phi(p) \mathcal{F}[u_\delta(s)](p) du_\delta (p,s) ds
 +\frac{r^2}{2\delta}  \int_0^t \int_{\sx } \sum_{i,j=1}^dQ_{ij}\p_{ij}\varphi(p)  G(p)^2  \, du_\delta (p,s) ds \\
&+ \int_0^t Error(s,\delta)ds
\end{split}
\end{equation}	
where
\begin{equation}\label{GrazingStep1Error}
 |Error(s,\delta)|\le \frac1\delta  C  \|D^3\varphi \|_\infty (\delta^3+r^3)
2+C \|D^2\phi\|_\infty \delta,
\end{equation}
and the constant $C$ is independent of $t$, $\phi$, $\delta$ and $r$.
\end{step}

\begin{proof}
First, for any test-function $\phi$ we have
\begin{align*}
\frac{d}{d\tau} \int_\sx \varphi(p) \,du_\delta (p,\tau)
=&\frac1\delta\frac{d}{dt}  \int_\sx \varphi(p) \,du  (p,t) \\
=&  \frac1\delta \int_{\sx^2} \mathbb{E}\Big[\varphi(p^*)-\varphi(p)\Big]
du (p,t)du (\tilde p,t).
\end{align*}
Performing a Taylor expansion up to the second order we have
$$\varphi(p^*)-\varphi(p)=\sum_{i=1}^d\p_i\varphi(p)(p_i^*-p_i)+
\frac{1}{2}\sum_{i,j=1}^d\p_{ij}\varphi(p)(p_i^*-p_i)(p_j^*-p_j)+R(p^*,p) $$
where
\begin{equation}\label{ErrorTerm}
|R(p^*,p)|\leq \frac16 \|D^3\varphi \|_\infty |p^*-p|^3.
\end{equation}
Thus,
\begin{equation}\label{Grazing20}
\begin{split}
& \int_{\sx^2} \mathbb{E}\Big[\varphi(p^*)-\varphi(p)\Big] du (p,t)du (\tilde p,t)  \\
 & \qquad =  \int_{\sx^2} \sum_{i=1}^d\p_i\varphi(p)\mathbb{E}\Big[p_i^*-p_i\Big] \\
  & \qquad \quad +\frac{1}{2} \sum_{i,j=1}^d\p_{ij}\varphi(p)
 \mathbb{E}\Big[(p_i^*-p_i)(p_j^*-p_j)\Big] du (p,t)du (\tilde p,t)  \\
    & \qquad  \quad   + \int_{\sx^2} \mathbb{E}\Big[R(p^*,p)\Big] du (p,t)du (\tilde p,t).
 \end{split}
\end{equation}

We examine each term in the right hand side. In view of the interaction rule
\eqref{interaccion}, namely
\begin{eqnarray*}
p^* - p = \delta \dm(p) f(\zeta,p)^TA f(\tilde \zeta,\tilde p) \Big(f(\zeta,p)-f(\tilde \zeta,\tilde p)\Big)  +  r(q -\vec{1}/d )G(p),
\end{eqnarray*}
we have
\begin{eqnarray*}
\mathbb{E}[p_i^* - p_i] =
\delta \dm(p) \mathbb{E}\Big[f(\zeta,p)^TA f(\tilde \zeta,\tilde p) \Big(f_i(\zeta,p)-f_i(\tilde \zeta,\tilde p)\Big)\Big]
+  \mathbb{E}[r(q_i -1/d )G(p)].
\end{eqnarray*}
If $q$ is a random variable uniformly distributed on $\sx$, then $\mathbb{E}[q_i]=1/d$. So,
in view of \eqref{Acumulada1} in the previous Lemma, we obtain
\begin{equation*}
\mathbb{E}[p_i^* - p_i] =
\delta \dm(p) \sum_{k=1}^d  a_{ik}(p_i  \tilde p_{k}  +   p_k\tilde p_{i} ).
\end{equation*}
By integrating we get
\begin{equation}\label{Grazing10}
\begin{split}
& \int_{\sx^2} \sum_{i=1}^d\p_i\varphi(p)\mathbb{E}\Big[p_i^*-p_i\Big]
du (p,t)du (\tilde p,t)  \\
& = \delta \int_{\sx} \dm(p) \sum_{i=1}^d\p_i\varphi(p)
\sum_{k=1}^d  a_{ik}(p_i \bar p_k(t) + p_k\bar p_i(t))du (p,t) \\
& = \delta \int_{\sx} \nabla\phi(p) \mathcal{F}[u_t](p) du (p,t),
\end{split}
\end{equation}
where the vector-field $\mathcal{F}$ is defined in \eqref{defcampo}.

We now study $\mathbb{E}\Big[(p_i^*-p_i)(p_j^*-p_j)\Big]$.
In view of the interaction rule,
\begin{align*}
(p_i^*-p_i)(p_j^*-p_j)
=&\Big[\delta \dm(p) f(\zeta)^TA f(\tilde \zeta) \Big(f_i(\zeta)-f_i(\tilde \zeta)\Big)  +  r(q_i -1/d )G(p)\Big]\\
&\quad \times \Big[\delta \dm(p) f(\zeta)^TA f(\tilde \zeta) \Big(f_j(\zeta)-f_j(\tilde \zeta)\Big)  +  r(q_j -1/d )G(p)\Big] \\
=&\Big(\delta \dm(p) f(\zeta)^TA f(\tilde \zeta)\Big)^2 \Big(f_i(\zeta)-f_i(\tilde \zeta)\Big)\Big(f_j(\zeta)-f_j(\tilde \zeta)\Big)\\
&+ \delta \dm(p)f(\zeta)^TA f(\tilde \zeta) \Big(f_i(\zeta)-f_i(\tilde \zeta)\Big)    r(q_j -1/d )G(p)\\
&+\delta \dm(p)f(\zeta)^TA f(\tilde \zeta) \Big(f_j(\zeta)-f_j(\tilde \zeta)\Big)     r(q_i -1/d )G(p)\\
&+ r^2(q_i -1/d )(q_j -1/d )G(p)^2.
\end{align*}
The second and third term have zero expected value since $q$ is independent of $\zeta$
and $\tilde \zeta$ and $\mathbb{E}[q_i]=1/d$. Moreover, in view of the definition
\eqref{defQ} of $Q$, the expectation of the last term is $r^2G(p)^2 Q_{ij}$. Lastly, since
$|f_i(\zeta,p)|\le 1$ for any $i=1,\ldots,d$ and any $p\in\sx$, the expectation of the first
term can be bounded by $C\delta^2$ for a constant $C$ depending only on $c$ and the
coefficients of $A$. Thus
\begin{equation*}
\begin{split}
\mathbb{E}\Big[(p_i^*-p_i)(p_j^*-p_j)\Big]
& = r^2G(p)^2 Q_{ij}  + O(\delta^2).
\end{split}
\end{equation*}
By integrating,
\begin{equation} \label{Grazing2}
\begin{split}
& \int_{\sx^2} \sum_{i,j=1}^d\p_{ij}\varphi(p)
\mathbb{E}\Big[(p_i^*-p_i)(p_j^*-p_j)\Big] du (p,t)du (\tilde p,t)  \\
& = r^2 \int_{\sx^2}G(p)^2 \sum_{i,j=1}^d\p_{ij}\varphi(p)  Q_{ij} du (p,t)
+ O(\delta^2)\|D^2\phi\|_\infty.
\end{split}
\end{equation}

It remains to bound the error term
$$ \int_{\sx^2} \mathbb{E}\Big[R(p^*,p)\Big] du (p,t)du (\tilde p,t)
\le \frac16 \|D^3\varphi \|_\infty \int_{\sx^2} \mathbb{E}\Big[|p^*-p|^3\Big] du (p,t) du
(\tilde p,t). $$ Using that $(a+b)^3\le C(a^3+b^3)$ for any $a,b\ge 0$, $|G(p)|\le 1$,
$|\dm(p)|\le c$, and $|f(\zeta)^TA f(\tilde \zeta)| \le \sum_{i,j=1}^d
|a_{ij}|f_i(\zeta,p)f_i(\tilde \zeta\tilde ,p) \le \sum_{i,j=1}^d |a_{ij}|$, we have
\begin{equation*}
\begin{split}
 |p^*-p|^3
& \le
\Big| \delta \dm(p) f(\zeta,p)^TA f(\tilde \zeta,\tilde p) \Big(f(\zeta,p)-f(\tilde \zeta,\tilde p)\Big)  +  r(q -\vec{1}/d )G(p) \Big|^3 \\
& \le C \delta^3 c^3 \Big(\sum_{i,j=1}^d |a_{ij}|\Big)^3 + Cr^3
= C(\delta^3+r^3).
\end{split}
\end{equation*}
Thus,
\begin{equation}\label{BoundError}
\begin{split}
\int_{\sx^2} \mathbb{E}\Big[R(p^*,p)\Big] du (p,t)du (\tilde p,t)
 \le C \|D^3\varphi \|_\infty (\delta^3+r^3).
\end{split}
\end{equation}

By inserting \eqref{Grazing10}, \eqref{Grazing2} and \eqref{BoundError} into
\eqref{Grazing20}, we obtain \eqref{Step1}.
\end{proof}

We now verify that the sequence $(u_\delta)$ verifies the assumptions of Arzela-Ascoli
Theorem, namely boundedness and uniform equicontinuity. The proof is based of the
previous Step. Since the error term \eqref{GrazingStep1Error} involves $\|\phi\|_{C^3}$, the
Wasserstein distance is of no use. Instead we will use the  norm $\|u\|_{sup}$, $u\in
P(\sx)$, defined by
\begin{equation}\label{normamedida3}
\|u\|_{sup}:=\sup  \Big\{ \int_\sx\varphi(p)\,du(p)     \text{ tal que  } \varphi\in C^3(\sx,\R) \text{ y } \|\varphi\|_3 \leq 1\Big\}.
\end{equation}
where the supremum is taken over all the function $\phi\in C^3(\sx)$ such that
$\|\varphi\|_3\le 1$ where
\begin{equation}\label{norma3}
\|\varphi\|_3:=\sum_{|\alpha|\leq 3}\|\partial^\alpha\varphi\|_\infty.
\end{equation}
According to \cite{gabetta1995metrics} this norm metricizes the weak convergence in $P(\sx)$.

\begin{step}\label{HypAA}
For any $t\in [0,T]$ and any $\delta>0$,
\begin{equation}\label{GrazingUnif}
 \|u_\delta(\cdot,\tau)\|_{sup}\le 1.
\end{equation}
Moreover there exists a constant $K>0$ such that for any $\tau,\tau'\in [0,T]$ and
any $r,\delta>0$ small,
\begin{equation}\label{GrazingEqui}
\| u_\delta (\cdot,\tau)- u_{\delta }(\cdot,\tau')\|_{sup}\leq K  \, |\tau-\tau'|.
\end{equation}
\end{step}

\begin{proof}
First for any $\phi\in C^3(\sx)$, $\|\phi\|_{C^3}\le 1$, recalling that
$u_\delta(.,\tau)$ is a probability measure, we clearly have
$$ \int_\sx\psi(p)\,du_\delta(p,\tau)
\leq\int_\sx |\psi(p)|\,du_\delta(p,\tau)
\le \int_\sx \,du_\delta(p,\tau) 1. $$
We deduce \eqref{GrazingUnif} taking the supremum over all such $\phi$.

To prove \eqref{GrazingEqui} we write using \eqref{Step1} that
for any $\phi\in C^3(\sc)$, $\|\phi\|_{C^3}\le 1$,
\begin{equation}\label{Grazing30}
\begin{split}
& \int_\sx  \varphi(p)\,du_\delta (p,\tau)
- \int_\sx  \varphi(p)\,du_\delta (p,\tau')	 \\
& =  \int_{\tau'}^\tau \int_{\sx} \nabla\phi(p) \mathcal{F}[u_s](p) du (p,s) ds
+\frac{r^2}{2\delta}  \int_{\tau'}^\tau \int_{\sx } \sum_{i,j=1}^dQ_{ij}\p_{ij}\varphi(p)  G(p)^2  \, du_\delta (p,s) ds \\
&+ \int_{\tau'}^\tau Error(s,\delta)ds
\end{split}
\end{equation}
Since $\sx$ is bounded we have $|\mathcal{F}[u_s](p)|\le C$. Thus
\begin{equation*}
\begin{split}
& \Big| \int_\sx  \varphi(p)\,du_\delta (p,\tau)
- \int_\sx  \varphi(p)\,du_\delta (p,\tau')	\Big|  \\
& \le C\Big( \|\nabla\phi\|_\infty + r^2/\delta \|D^2\varphi \|_\infty
+ \frac1\delta  \|D^3\varphi \|_\infty (\delta^3+r^3)\Big) |\tau-\tau'| \\
& \le C(1+r^2/\delta).
\end{split}
\end{equation*}
Since we assumed that $r^2/\delta\to \lambda$, we obtain \eqref{GrazingEqui}.
\end{proof}

We fix $T>0$. The space $P(\sx)$ is compact for the weak convergence (and so for the norm
$\|.\|_{sup}$). In view of the previous Step we can apply Arzela-Ascoli Theorem to the
sequence of continuous functions $u_\delta: [0,T]\to\mathcal{P}(\sx)$ to obtain that a
subsequence converges uniformly as $\delta \to 0$.

A diagonal argument shows in fact that a subsequence, that we still denote $(u_\delta)$, converges
in  $C([0,T],\mathcal{P}(\sx))$ for any $T>0$ to some $v\in C([0,+\infty),\mathcal{P}(\sx))$.

We now verify that $v$ satisfies \eqref{DefWeakSol}.

\begin{step}\label{GrazingLimit}
$v$ satisfies \eqref{DefWeakSol}.
\end{step}

\begin{proof}
We need to pass to the limit in \eqref{Step1} as $\delta\to 0$. We fix some $\phi\in
C^3(\sx)$ and $t>0$, and recall that $r^2/\delta \to \lambda$. Then it is easily seen the last
term in the right hand side of  \eqref{Step1} can be bounded by $ C\|\phi\|_{C^3}t (\delta +
r.r^2/\delta) \to 0$. Let us pass to the limit in the second term in the right hand side. For a fixed
$s\in [0,t]$, we write
\begin{equation}\label{Grazing100}
\begin{split}
& \int_{\sx} \nabla\phi(p) \mathcal{F}[u_\delta(s)](p) du_\delta (p,s) ds \\
&\qquad =  \int_{\sx} \nabla\phi(p) (\mathcal{F}[u_\delta(s)](p)-\mathcal{F}[u(s)](p)) du_\delta (p,s) ds   \\
&\qquad \quad + \int_{\sx} \nabla\phi(p) \mathcal{F}[u(s)](p) du_\delta (p,s) ds.
\end{split}
\end{equation}
Notice that
$$ \mathcal{F}[u_\delta(s)](p)-\mathcal{F}[u(s)](p)
= h(p)(p_ie_i^T A m^\delta(s)+m^\delta_i(s)e_i^TAp) $$
where
$$ m^\delta(s)=\int_\sx \tilde p\,d(u_\delta(s,\tilde p)-u(s,\tilde p)). $$
Since $u_\delta(s)\to v(s)$ weakly uniformly in $s\in [0,t]$, we also have that
$$W_1(u_\delta(s),v(s))\to 0$$ uniformly in $s\in [0,t]$. Thus $m^\delta(s)\to 0$ and then
$\mathcal{F}[u_\delta(s)](p)\to\mathcal{F}[u(s)](p)$ uniformly in $p\in \sx$. We deduce
that the first term in the right hand side of \eqref{Grazing100} goes to 0. The second term
also goes to 0 since
 $\nabla\phi$ and $\mathcal{F}[u(s)]$ are Lipschitz.

Moreover,
$$ \Big| \int_{\sx} \nabla\phi(p) \mathcal{F}[u_\delta(s)](p) du_\delta (p,s) \Big|
\le \|\nabla\phi\|_\infty \|\mathcal{F}[u_\delta(s)]\|_\infty
\le C.$$

We then conclude that
$$ \int_0^t \int_{\sx} \nabla\phi(p) \mathcal{F}[u_\delta(s)](p) du_\delta (p,s) ds
\to \int_0^t \int_{\sx} \nabla\phi(p) \mathcal{F}[u(s)](p) du(p,s) ds
$$
applying the Dominated Convergence Theorem. We can prove in the same way that the second
term in the right hand side of \eqref{Step1} converges to
$$ \lambda \int_0^t \int_{\sx } \sum_{i,j=1}^dQ_{ij}\p_{ij}\varphi(p)  G(p)^2
 \, du(p,s) ds. $$
\end{proof}

To conclude the proof it remains to study the case where $r^2/\delta^\alpha\to \lambda >0$
for some $\alpha\ in (0,1)$. In that case we rescale time considering $\tau=\delta^\alpha t$.
Reasoning as before we can write

\begin{equation*}\label{igualdadtaylor2}
  \begin{split}
  & \int_\sx  \varphi(p)\,du_\delta (p,\tau)- \int_\sx  \varphi(p)\, du_{\delta }(p,\tau') \\
 &\qquad  =\int_{\tau '}^\tau  \frac{d}{d\tau} \int_\sx \varphi(p)\,du_\delta (p,s)ds \\
  &\qquad =\frac{1}{\delta^\alpha}  \int_{\tau'}^\tau  \int \mathbb{E}[\varphi(p^*)-\varphi(p)]
       du_\delta (p,s)du_\delta (\tilde p,s) \\
 & \qquad= \delta^{1-\alpha} \int_{\tau'}^\tau  \int_\sx \nabla \phi \mathcal{F}[u_\delta(s)]
                            \,du_\delta (p,s)
 +\frac{r^2}{2\delta^\alpha}   \int_{\sx } \sum_{i,j=1}^dQ_{ij}\p_{ij}\varphi(p)  G(p)^2  \, du_\delta (p,s)  \\
 & \qquad \quad +\frac{1}{\delta^\alpha}\int_{\sx^2}  \mathbb{E}[R(p^*,p)] du_\delta (p,s)du_\delta (\tilde p,s).
  \end{split}
\end{equation*}
In view of \eqref{BoundError} and recalling that we assume $r^2/\delta^\alpha\to \lambda$,
the last term can be bounded by
$$ C \|\phi\|_{C^3} (\delta^{3-\alpha}+r.r^2/\delta^\alpha)
+ C\delta^{2-\alpha} \|D^2\phi\|_\infty \le C \|\phi\|_{C^3} (\delta+r).  $$ It follows that
Step \ref{HypAA} still holds and thus, applying Arzela-Ascoli Theorem, we obtain  that a
subsequence, that we still denote $(u_\delta)$, converges in  $C([0,T],\mathcal{P}(\sx))$ for
any $T>0$ to some $v\in C([0,+\infty),\mathcal{P}(\sx))$. Passing to the limit $\delta\to 0$
as in Step \ref{GrazingLimit}, we deduce that $v$ satisfies
\begin{align*}
  \int_\sx  \varphi(p) \,dv (p,\tau) = & \int_\sx  \varphi(p)\, dv(p,\tau') +\int_{\tau'}^\tau
       \frac{\lambda}{2 }   \int_{\sx } \sum_{i,j=1}^dQ_{ij}\p_{ij}\varphi(p)  G(p)^2  \, dv (p,s)ds
\end{align*}
which is the weak formulation of
\begin{equation*}
 \frac{d}{d\tau}v=\frac{\lambda }{2} \sum_{i,j=1}^dQ_{ij}\partial_{ij}(G^2\, v).
\end{equation*}
This concludes the proof of Theorem \ref{grazing}.

\subsection{Well-posedness of equation \eqref{LimitEqq}. Proof of Theorem \ref{teotransporte}.}

Let us fix some $v\in C([0,+\infty],P(\sx))$ and denote $\mathcal{F}(t,p):=\mathcal{F}[v(t)](p)$, namely
$$ \F_i(t,p)= \dm(p) (p_i e_i^TA \pp(t) + \pp_i(t)e_i^T A p), \qquad
\pp(t)=\int_\sx p\,dv_t(p). $$ Notice that $\vec{1}:=(1,\ldots,1)\in\R^d$ is normal to $\sx$ and
that for any $p\in\sx$,
$$ \vec{1}.\mathcal{F}(t,p) = \sum_{i=1}^d \mathcal{F}_i(t,p)
= \dm(p)(p.A\bar p(t)+\bar p(t).Ap) = 0 $$
since $A$ is antisymmetric, and
$$ \F(t,p)=0 \qquad \text{for any $p\in \p\sx$} $$
since $\dm(p)=0$ if $p\in\p\sx$. It follows that any integral curve of $\mathcal{F}(t,p)$
starting from a point in $\sx$ stays in the compact $\sx$ forever. Let us denote $T_{s,t}^v$
the flow of the vector field $\mathcal{F}(t,p)$ namely
$$ \frac{d}{dt}T_{s,t}^v(p) = \mathcal{F}[v(t)](T_{s,t}^v(p)), \qquad
T_{s,s}t^v(p)=p. $$
We also let $T_t^v(p):=T_{0,t}^v(p)$.
We then have

\begin{lem}There holds
$$ T_t^v(p)\in\sx \qquad \text{for any $p\in\sx$ and any $t\ge 0$.} $$
\end{lem}

The characteristic method allows to show in a standard way that the equation
$$\frac{d}{dt}u+\text{div}(\F[v(t)](p) u)=0$$
with initial condition $u_0\in P(\sx)$ has a unique solution in $C([0,+\infty],P(\sx))$ given by
$u(t)=\T^v_t\sharp u_0$.

Thus proving the existence of a unique solution  to
\begin{equation}\label{FixePoint20}
\frac{d}{dt}v+\text{div}(\F[v(t)](p) v)=0
\end{equation}
for a given initial condition $v_0\in P(\sx)$ is equivalent to proving the existence of a solution to the fixed-point equation
\begin{equation}\label{FixedPoint10}
v(t)=T_t^v\sharp v_0.
\end{equation}
This can be done applying the Banach fixed-point theorem to the map $\Gamma(v)_t:=
T_t^v\sharp v_0$ in the complete metric space $ \{v\in C([0,T],P(\sx)),\, v(0)=v_0\} $ for a small
enough $T>0$ depending on $\|v_0\|_{TV}$. The details of the proof are an adaptation of
e.g.  \cite{canizo2011well}. The adaptation is almost straight forward noticing that
$\mathcal{F}[v(t)](p)$ is bounded Lipschitz in $p$ uniformly in $t$ for a given $v\in
C([0,T],P(\sx))$. We can then repeat the process on $[T,2T]$, $[2T,3T]$,\dots to obtain $v\in
C([0,+\infty],P(\sx))$ satisfying \eqref{FixedPoint10} which is thus the unique solution to
\eqref{FixePoint20} with initial condition $v_0$. The continuity with respect to initial
conditions can also be obtained following the proof of \cite{canizo2011well}.

\bibliography{biblio}
\bibliographystyle{plain}

\end{document}